\documentclass[preprint,12pt]{elsarticle}

\usepackage{amssymb}
\usepackage{amsmath}
\usepackage{amsthm}

\newtheorem{theorem}{Theorem}[section]

\newtheorem{corollary}[theorem]{Corollary}

\theoremstyle{definition} %text of this environment is typesetted in roman letters

\newtheorem{example}{Example}[section]

\newcommand\on{\operatorname}
\newcommand\R{\on{Ric}}
\newcommand\sign{\on{sign}}

\begin{document}

\begin{frontmatter}



\title{Flat $3$-manifolds with diagonal metrics and applications to warped products} %% Article title


\author{Adara M. Blaga (corresponding author)} 

\affiliation{organization={Department of Mathematics, West University of Timisoara},
            addressline={V. P\^{a}rvan 4},
            city={Timi\c{s}oara},
            postcode={300223},
            country={Romania}}
            
\author{Dan Radu La\c tcu} 

\affiliation{organization={Central University Library of Timisoara},
            addressline={V. P\^{a}rvan 4A},
            city={Timi\c{s}oara},
            postcode={300223},
            country={Romania}}

%% Abstract
\begin{abstract}
We provide necessary and sufficient conditions for a $3$-dimensional submanifold of $\mathbb R^3$ endowed with a diagonal metric to be flat. As applications, we characterize the flat manifolds of warped product-type, more precisely, the warped, biwarped, sequential warped, and doubly warped product manifolds, and we state the corresponding nonexistence results.
\end{abstract}



%% Keywords

\end{frontmatter}



\section{Preliminaries}

Let $(M,g)$ be an $n$-dimensional Riemannian manifold, and let
\begin{align*}
R(X,Y)Z&:=\nabla_X\nabla_YZ-\nabla_Y\nabla_XZ-\nabla_{[X,Y]}Z,\\
\R(Y,Z)&:=\sum_{k=1}^ng(R(E_k,Y)Z,E_k)
\end{align*}
be the Riemannian and the Ricci curvature tensor fields of the metric $g$, where $\nabla$ is the Levi-Civita connection of $g$, and $\{E_1,\dots,E_n\}$ is a local orthonormal frame on $(M, g)$.
We recall that a Riemannian manifold $(M, g)$ is called \textit{flat} if $R=0$, and it is called \textit{Ricci-flat} if $\R=0$.

Recently, in \cite{bl22}, we have described some $3$-dimensional almost $\eta$-Ricci solitons with diagonal metrics, providing also conditions for the manifold to be flat. The aim of the present paper is to characterize the flat $3$-dimensional Riemannian submanifolds of $\mathbb R^3$ with a diagonal metric, with a special view towards warped products.
As applications, we characterize the flat manifolds of warped product-type, more exactly, the warped \cite{bishop}, biwarped \cite{hakan} (a particular case of multiply warped \cite{nolker}), sequential warped \cite{ude}, and doubly warped \cite{erlich} product manifolds, and we prove some nonexistence results.
It is worth to be mentioned the importance and the applicability in Physics, especially in the Theory of Relativity, of (semi-)Riemannian manifolds of warped product-type. For example, the standard spacetime models such as Robertson--Walker, Schwarzschild, static and Kruskal, are all warped product manifolds. Moreover, the simplest models of neighborhoods of stars and black holes are also warped products (for details, see, for instance, \cite{neil}).

We shall briefly recall their definitions.
Let $(M_i,g_i)$, $i\in \{1,2\}$, be two Riemannian manifolds, let $M:=M_1\times M_2$, $\pi:M\rightarrow M_1$ be the canonical projection, and let
$f:\nolinebreak M_1\rightarrow\nolinebreak \mathbb R \setminus \{0\}$ be a smooth function. In 1969, Bishop and O'Neill introduced the notion of warped product manifold. More precisely, $(M,g)=:M_1\times _f M_2$ is called a \textit{warped product manifold} \cite{bishop} if the Riemannian metric $g$ is given by
$$g=\pi_{1}^{*}(g_1)+(\pi_1^*(f_1))^2 \pi_{2}^{*}(g_2).$$

This notion has been further extended to a larger number of manifolds.

If $f_1:M_1\rightarrow \mathbb R \setminus \{0\}$ and $f_2:M_2\rightarrow \mathbb R \setminus \{0\}$ are two smooth functions, Ehrlich called the manifold $(M,g)=:_{f_2}M_1\times_{f_1}  M_2$ a \textit{doubly warped product manifold} \cite{erlich} if the Riemannian metric $g$ is given by
$$g=(\pi_2^*(f_2))^2 \pi_{1}^{*}(g_1)+(\pi_1^*(f_1))^2 \pi_{2}^{*}(g_2).$$

In the case of three manifolds, the definitions of the manifolds of warped product-type mentioned above are the following.
Let $(M_i,g_i)$, $i\in \{1,2,3\}$, be three Riemannian manifolds, $M:=\nolinebreak M_1\times M_2\times M_3$, $\pi_i:M\rightarrow M_i$, $i\in \{1,2,3\}$, be the canonical projections. Then,

(1) $(M,g)=:M_1\times_{f_1}  M_2 \times_{f_2} M_3$ is called a \textit{biwarped product manifold} \cite{hakan} if
$$g=\pi_{1}^{*}(g_1)+(\pi_1^*(f_1))^2 \pi_{2}^{*}(g_2)+(\pi_1^*(f_2))^2 \pi_{3}^{*}(g_3),$$
where $f_1,f_2:M_1\rightarrow \mathbb R \setminus \{0\}$;

(2) $(M,g)=:(M_1\times_{f_1}  M_2) \times_{f_2} M_3$ is called a \textit{sequential warped product manifold} \cite{ude} if
$$g=\pi_{1}^{*}(g_1)+(\pi_1^*(f_1))^2 \pi_{2}^{*}(g_2)+(\pi^*(f_2))^2 \pi_{3}^{*}(g_3),$$
where $f_1:M_1\rightarrow \mathbb R \setminus \{0\}$ and $f_2:M_1\times M_2\rightarrow \mathbb R \setminus \{0\}$.

We shall further use the same notation for a function and its pull-back as well as for a metric and its pull-back on the product manifold. Also, we shall say that the above manifolds are \textit{proper} when all the functions are nonconstant.

\bigskip

Let $I_i\subseteq \mathbb R$, $i\in \{1,2,3\}$, be three open intervals and let $I=I_1\times I_2\times I_3$. We consider ${g}$ a Riemannian metric on $I$ given by
\begin{equation} \label{f6}
{g}=\frac{1}{f_1^2}(dx^1)^2+\frac{1}{f_2^2}(dx^2)^2+\frac{1}{f_3^2}(dx^3)^2,
\end{equation}
where $f_1$, $f_2$, and $f_3$ are smooth functions nowhere zero on $I$, and $x^1,x^2,x^3$ stand for the standard coordinates in $\mathbb R^3$. Let
$$\Big\{E_1:=f_1\frac{\partial}{\partial x^1}, \ \ E_2:=f_2\frac{\partial}{\partial x^2}, \ \ E_3:=f_3\frac{\partial}{\partial x^3}\Big\}$$
be a local orthonormal frame. We will denote as follows:
$$\frac{f_2}{f_1}\cdot\frac{\partial f_1}{\partial x^2}=:a_{12}, \ \ \frac{f_3}{f_1}\cdot\frac{\partial f_1}{\partial x^3}=:a_{13}, \ \ \frac{f_1}{f_2}\cdot\frac{\partial f_2}{\partial x^1}=:a_{21},$$$$\frac{f_3}{f_2}\cdot\frac{\partial f_2}{\partial x^3}=:a_{23}, \ \ \frac{f_1}{f_3}\cdot\frac{\partial f_3}{\partial x^1}=:a_{31}, \ \ \frac{f_2}{f_3}\cdot\frac{\partial f_3}{\partial x^2}=:a_{32}.$$
Computing the Lie brackets, defined as $[X,Y](h):=X(Y(h))-Y(X(h))$ for any vector fields $X$, $Y$ and any smooth function $h$ on $I$, we find:
$$[E_1,E_2]=-a_{12}E_1+a_{21}E_2=-[E_2,E_1],$$
$$[E_2,E_3]=-a_{23}E_2+a_{32}E_3=-[E_3,E_2],$$
$$[E_3,E_1]=-a_{31}E_3+a_{13}E_1=-[E_1,E_3].$$
On the base vector fields, the Levi-Civita connection $\nabla$ of $g$, obtained from
the Koszul's formula,
\begin{align*}
2g(\nabla_XY,Z)&=X(g(Y,Z))+Y(g(Z,X))-Z(g(X,Y))\\
&\hspace{12pt}-g(X,[Y,Z])+g(Y,[Z,X])+g(Z,[X,Y]),
\end{align*}
is given by:
$$\nabla_{E_1}E_1=a_{12}E_2+a_{13}E_3, \ \ \nabla_{E_2}E_2=a_{21}E_1+a_{23}E_3, \ \ \nabla_{E_3}E_3=a_{31}E_1+a_{32}E_2,$$
$$\nabla_{E_1}E_2=-a_{12}E_1, \ \ \nabla_{E_2}E_3=-a_{23}E_2, \ \ \nabla_{E_3}E_1=-a_{31}E_3,$$
$$\nabla_{E_1}E_3=-a_{13}E_1, \ \ \nabla_{E_3}E_2=-a_{32}E_3, \ \ \nabla_{E_2}E_1=-a_{21}E_2.$$

In the rest of the paper, whenever a function $f$ on $I\subseteq \mathbb R^3$ depends only on some of its variables, we will write in its argument only that variables in order to emphasize this fact, for example, $f(x^i)$, $f(x^i,x^j)$. Also, whenever a coefficient function $f_i$ of the metric $g$, $i\in \{1,2,3\}$, depends only on one of its variables, we will denote by $h_i:=\frac{\displaystyle f_i'}{\displaystyle f_i}$.

\section{The flatness condition}

The Riemannian and the Ricci curvature tensor fields of $(I,g)$, where $g$ is given by (\ref{f6}), are:
\begin{align*}
R(E_1,E_2)E_2&=[E_1(a_{21})+E_2(a_{12})-a_{21}^2-a_{12}^2-a_{13}a_{23}]E_1\\
&\hspace{12pt}+[E_1(a_{23})+a_{21}(a_{13}-a_{23})]E_3,\\
R(E_2,E_1)E_1&=[E_1(a_{21})+E_2(a_{12})-a_{21}^2-a_{12}^2-a_{13}a_{23}]E_2\\
&\hspace{12pt}+[E_2(a_{13})+a_{12}(a_{23}-a_{13})]E_3,\\
R(E_1,E_3)E_3&=[E_1(a_{31})+E_3(a_{13})-a_{31}^2-a_{13}^2-a_{12}a_{32}]E_1\\
&\hspace{12pt}+[E_1(a_{32})+a_{31}(a_{12}-a_{32})]E_2,\\
R(E_2,E_3)E_3&=[E_3(a_{31})+a_{32}(a_{21}-a_{31})]E_1\\
&\hspace{12pt}+[E_2(a_{32})+E_3(a_{23})-a_{32}^2-a_{23}^2-a_{21}a_{31}]E_2,\\
R(E_3,E_1)E_1&=[E_3(a_{12})+a_{13}(a_{32}-a_{12})]E_2\\
&\hspace{12pt}+[E_1(a_{31})+E_3(a_{13})-a_{31}^2-a_{13}^2-a_{12}a_{32}]E_3,\\
R(E_3,E_2)E_2&=[E_3(a_{21})+a_{23}(a_{31}-a_{21})]E_1\\
&\hspace{12pt}+[E_2(a_{32})+E_3(a_{23})-a_{32}^2-a_{23}^2-a_{21}a_{31}]E_3,\\
R(E_1,E_2)E_3&=[E_2(a_{13})+a_{12}(a_{23}-a_{13})]E_1\\
&\hspace{12pt}-[E_1(a_{23})+a_{21}(a_{13}-a_{23})]E_2,\\
R(E_2,E_3)E_1&=[E_3(a_{21})+a_{23}(a_{31}-a_{21})]E_2\\
&\hspace{12pt}-[E_3(a_{31})+a_{32}(a_{21}-a_{31})]E_3,\\
R(E_3,E_1)E_2&=-[E_3(a_{12})+a_{13}(a_{32}-a_{12})]E_1\\
&\hspace{12pt}+[E_1(a_{32})+a_{31}(a_{12}-a_{32})]E_3,
\end{align*}
\begin{align*}
\R(E_1,E_1)&=E_1(a_{21})+E_1(a_{31})+E_2(a_{12})+E_3(a_{13})\\
&\hspace{12pt}-a_{21}^2-a_{12}^2-a_{31}^2-a_{13}^2-a_{12}a_{32}-a_{13}a_{23},\\
\R(E_2,E_2)&=E_1(a_{21})+E_2(a_{12})+E_2(a_{32})+E_3(a_{23})\\
&\hspace{12pt}-a_{21}^2-a_{12}^2-a_{32}^2-a_{23}^2-a_{13}a_{23}-a_{21}a_{31},\\
\R(E_3,E_3)&=E_1(a_{31})+E_2(a_{32})+E_3(a_{13})+E_3(a_{23})\\
&\hspace{12pt}-a_{31}^2-a_{13}^2-a_{32}^2-a_{23}^2-a_{12}a_{32}-a_{21}a_{31},\\
\R(E_1,E_2)&=E_1(a_{32})+a_{31}(a_{12}-a_{32}),\\
\R(E_1,E_3)&=E_1(a_{23})+a_{21}(a_{13}-a_{23}),\\
\R(E_2,E_3)&=E_2(a_{13})+a_{12}(a_{23}-a_{13}).
\end{align*}

We aim to determine conditions that the three functions, $f_1$, $f_2$, and $f_3$, must satisfy for the Riemannian manifold $(I,g)$ to be flat. Let us firstly remark that,
if $f_i:I\rightarrow \mathbb R\setminus \{0\}$, $f_i=f_i(x^i)$ for $i\in \{1,2,3\}$ (in particular, if they are constant), then $(I,g)$ is a flat Riemannian manifold.

\bigskip

Since, in dimension $3$, every Ricci-flat manifold is also flat, we shall make use of the vanishing of the Ricci tensor in proving the following results.

\begin{theorem}\label{ps1b}
If $f_i=f_i(x^1)$ for $i\in \{1,2,3\}$, then $(I,g)$ is a flat Riemannian manifold if and only if one of the following assertions holds:

(1) $f_2=k_2\in \mathbb R\setminus \{0\}$ and $f_3=k_3\in \mathbb R\setminus \{0\}$;

(2) $f_2=k_2\in \mathbb R\setminus \{0\}$ and $f_3=\frac{\displaystyle k_3}{\displaystyle F-c_3}$, where $F$ is an antiderivative of $\frac{\displaystyle 1}{\displaystyle f_1}$, $k_3\in \mathbb R \setminus \{0\}$, and $c_3\in \mathbb R\setminus \{F(x^1)|\ x^1\in I_1\}$;

(3) $f_2=\frac{\displaystyle k_2}{\displaystyle F-c_2}$ and $f_3=k_3\in \mathbb R\setminus \{0\}$, where $F$ is an antiderivative of $\frac{\displaystyle 1}{\displaystyle f_1}$, $k_2\in \mathbb R \setminus \{0\}$, and $c_2\in \mathbb R\setminus \{F(x^1)|\ x^1\in I_1\}$.
\end{theorem}
\begin{proof}
We have
$$a_{i1}=f_1\frac{\displaystyle f_i'}{\displaystyle f_i} \ \ \textrm{for} \ \ i\in\{2,3\}$$
and
$$a_{ij}=0 \ \ \textrm{for} \ \ (i,j)\in\{(1,2),(1,3),(2,3),(3,2)\},$$
and we get:
\begin{align*}
\R(E_1,E_1)&=E_1(a_{21})-a_{21}^2+E_1(a_{31})-a_{31}^2,\\
\R(E_2,E_2)&=E_1(a_{21})-a_{21}^2-a_{21}a_{31},\\
\R(E_3,E_3)&=E_1(a_{31})-a_{31}^2-a_{21}a_{31},\\
\R(E_1,E_2)&=\R(E_1,E_3)=\R(E_2,E_3)=0.
\end{align*}

Then, $\R=0$ if and only if
\begin{equation*}
\left\{
    \begin{aligned}
&      f_1\Big(f_1\frac{\displaystyle f_2'}{\displaystyle f_2}\Big)'-\Big(f_1\frac{\displaystyle f_2'}{\displaystyle f_2}\Big)^2+ f_1\Big(f_1\frac{\displaystyle f_3'}{\displaystyle f_3}\Big)'-\Big(f_1\frac{\displaystyle f_3'}{\displaystyle f_3}\Big)^2=0\\
&f_1\Big(f_1\frac{\displaystyle f_2'}{\displaystyle f_2}\Big)'-\Big(f_1\frac{\displaystyle f_2'}{\displaystyle f_2}\Big)^2-f_1^2\frac{\displaystyle f_2'}{\displaystyle f_2}\cdot\frac{\displaystyle f_3'}{\displaystyle f_3}=0\\
&f_1\Big(f_1\frac{\displaystyle f_3'}{\displaystyle f_3}\Big)'-\Big(f_1\frac{\displaystyle f_3'}{\displaystyle f_3}\Big)^2-f_1^2\frac{\displaystyle f_2'}{\displaystyle f_2}\cdot\frac{\displaystyle f_3'}{\displaystyle f_3}=0
    \end{aligned}
  \right. ,
\end{equation*}
and the previous system becomes
\begin{equation*}
\left\{
    \begin{aligned}
&      h_2'-h_2^2+h_1h_2+h_3'-h_3^2+h_1h_3=0\\
& h_2'-h_2^2+h_1h_2-h_2h_3=0\\
& h_3'-h_3^2+h_1h_3-h_2h_3=0
    \end{aligned}
  \right. ,
\end{equation*}
which is equivalent to
\begin{equation*}
\left\{
    \begin{aligned}
&      h_2h_3=0\\
& h_2'=h_2(h_2-h_1)\\
& h_3'=h_3(h_3-h_1)
    \end{aligned}
  \right. .
\end{equation*}
If $h_i\neq 0$, where $i\in \{2, 3\}$, let $J_1$ be a maximal open subinterval in $I_1$ such that $h_i(x^1)\neq 0$ everywhere on $J_1$. From $h_i'=h_i(h_i-h_1)$, we get $\frac{\displaystyle h_i'}{\displaystyle h_i}=h_i-h_1=\frac{\displaystyle f_i'}{\displaystyle f_i}-\frac{\displaystyle f_1'}{\displaystyle f_1}$ on $J_1$, and, by integration, we infer that $\frac{\displaystyle f_i'}{\displaystyle f_i}=h_i=d_i\frac{\displaystyle f_i}{\displaystyle f_1}$ on $J_1$, with $d_i\in \mathbb R \setminus \{0\}$, from which, $f_i=\frac{\displaystyle k_i}{\displaystyle F-c_i}$, and $h_i=\frac{\displaystyle -1}{\displaystyle f_1(F-c_i)}$ on $J_1$, where $F=F(x^1)$ is an antiderivative of $\frac{\displaystyle 1}{\displaystyle f_1}$, $k_i\neq 0$, and $c_i\in \mathbb R$ such that $F(x^1)-c_i\neq 0$ for any $x^1\in J_1$. Due to the maximality of $J_1$ and the continuity of $h_i$, we get $J_1=I_1$; hence, $h_i(x^1)\neq 0$ for any $x^1\in I_1$.

Therefore, we get either $h_i=0$ (i.e., $f_i$ constant) for $i\in \{2,3\}$, or  $h_i=0$ and $h_j'=h_j(h_j-h_1)$, $h_j\neq 0$ everywhere, for $(i,j)\in \{(2,3), (3,2)\}$, which leads to $f_j= \displaystyle\frac{k_j}{F-c_j}$, where $F=F(x^1)$ is an antiderivative of $\frac{\displaystyle 1}{\displaystyle f_1}$, $k_j\in \mathbb R \setminus \{0\}$, and $c_j\in \mathbb R$ such that $F(x^1)\neq c_j$ for any $x^1\in I_1$.
\end{proof}

\begin{corollary}\label{ps1a}
Under the hypotheses of Theorem \ref{ps1b}, if $h_1=c_0\in\mathbb R\setminus \{0\}$, then $(I,g)$ is a flat Riemannian manifold if and only if either both of the functions $f_2$ and $f_3$ are constant or
$$f_i=k_i\in \mathbb R\setminus \{0\} \ \ \textrm{and} \ \ f_j(x^1)= \frac{k_j}{e^{-c_0x^1}-c_j} \ \ \textrm{for} \ (i,j)\in \{(2,3), (3,2)\},$$
where $k_j\in \mathbb R\setminus \{0\}$, $c_j\in \mathbb R\setminus\{e^{-c_0x^1}|\ x^1\in I_1\}$. Moreover, $f_1(x^1)=c_1e^{c_0x^1}$, with  $c_1\in \mathbb R\setminus \{0\}$.

In particular, the statement is valid for $I_1=\mathbb R$ if we only add the condition $c_j\leq 0$.
\end{corollary}
\begin{proof}
From Theorem \ref{ps1b}, we deduce that either both of the functions $f_2$ and $f_3$ are constant or
$$f_i=k_i\in \mathbb R\setminus \{0\} \ \ \textrm{and} \ \ h_j'=h_j(h_j-c_0), h_j\neq 0, \ \textrm{for} \ (i,j)\in \{(2,3), (3,2)\}.$$
If $\frac{\displaystyle f_j'}{\displaystyle f_j}=h_j=c_0\neq 0$, then $f_j(x^1)=c_je^{c_0x^1}$, with
$c_j\in \mathbb R\setminus \{0\}$. If $h_j\neq 0$ and $h_j\neq c_0$, hence $h_j$ is different from $0$ and $c_0$ in any point of a maximal open interval $J_1\subseteq I_1$, then
$$1=\frac{h_j'}{h_j(h_j-c_0)}=\frac{h_j'}{c_0(h_j-c_0)}-\frac{h_j'}{c_0h_j}=\frac{1}{c_0}\Big[\frac{(h_j-c_0)'}{h_j-c_0}-\frac{h_j'}{h_j}\Big],$$
which, by integration, gives
$$\frac{f_j'(x^1)}{f_j(x^1)}=h_j(x^1)=\frac{c_0}{1-c_je^{c_0x^1}}$$ on $J_1$, where $c_j\in \mathbb R\setminus \{0\}$, $c_je^{c_0x^1}\neq 1$
everywhere on $J_1$. Since $J_1$ is maximal, we deduce that $h_j$ has the above expression on $I_1$. So,
$$f_j(x^1)=\frac{k_j}{e^{-c_0x^1}-c_j}$$
on $I_1$, with $k_j\in \mathbb R\setminus \{0\}$, $c_je^{c_0x^1}\neq 1$ everywhere on $I_1$.
\end{proof}

\begin{corollary}\label{ps1c}
Under the hypotheses of Theorem \ref{ps1b}, if $h_1=0$, then $(I,g)$ is a flat Riemannian manifold if and only if either both of the functions $f_2$ and $f_3$ are constant or
$$f_i=k_i\in \mathbb R\setminus \{0\} \ \ \textrm{and} \ \ f_j(x^1)=\frac{k_j}{x^1-c_j} \ \ \textrm{for} \ (i,j)\in \{(2,3), (3,2)\},$$
where $k_j\in \mathbb R\setminus \{0\}$, $c_j\in \mathbb R\setminus I_1$. Moreover, $f_1(x^1)=k_1\in \mathbb R\setminus \{0\}$.
\end{corollary}
\begin{proof}
From Theorem \ref{ps1b}, we deduce that either both of the functions $f_2$ and $f_3$ are constant or
$f_i=k_i\in \mathbb R\setminus \{0\}$ and $f_j=\frac{\displaystyle a_j}{\displaystyle F-b_j}$ for $(i,j)\in\nolinebreak \{(2,3), (3,2)\}$, where $F=F(x^1)$ is an antiderivative of $\frac{\displaystyle 1}{\displaystyle f_1}$, $a_j\in \mathbb R \setminus \{0\}$ and $b_j\in \mathbb R$ such that $F(x^1)-b_j\neq 0$ for any $x^1\in I_1$.

From $h_1=0$, we deduce that $f_1$ is constant, $F(x^1)=a_1x^1+b_1$, with $a_1\neq 0, b_1\in \mathbb R$, hence $f_j(x^1)=\displaystyle\frac{k_j}{x^1-c_j}$, with $k_j\in \mathbb R\setminus \{0\}, c_j\in \mathbb R$ such that $c_j\notin I_1$.
\end{proof}

From the previous corollary, we deduce
\begin{corollary}
Under the hypotheses of Theorem \ref{ps1b}, if $h_1=0$, then $(\mathbb R^3,g)$ is a flat Riemannian manifold if and only if all the functions $f_1$, $f_2$, and $f_3$ are constant.
\end{corollary}

\begin{theorem}\label{ps11n}
If $f_1=f_1(x^2)$, $f_2=f_2(x^3)$, $f_3=f_3(x^1)$, then $(I,g)$ is a flat Riemannian manifold if and only if one of the following assertions holds:

(1) $f_1=k_1\in\mathbb R\setminus \{0\}$, $f_2=k_2\in\mathbb R\setminus \{0\}$, and $f_3=k_3\in\mathbb R\setminus \{0\}$;

(2) $f_1=k_1\in\mathbb R\setminus \{0\}$, $f_2=k_2\in\mathbb R\setminus \{0\}$, and $f_3(x^1)=\frac{\displaystyle k_3}{\displaystyle x^1-c_3}$, with $k_3\in\mathbb R\setminus \{0\}$, $c_3\in\mathbb R\setminus I_1$;

(3) $f_1(x^2)=\frac{\displaystyle k_1}{\displaystyle x^2-c_1}$, $f_2=k_2\in\mathbb R\setminus \{0\}$, and $f_3=k_3\in\mathbb R\setminus \{0\}$, with $k_1\in\mathbb R\setminus \{0\}$, $c_1\in\mathbb R\setminus I_2$;

(4) $f_1=k_1\in\mathbb R\setminus \{0\}$, $f_2(x^3)=\frac{\displaystyle k_2}{\displaystyle x^3-c_2}$, and $f_3=k_3\in\mathbb R\setminus \{0\}$, with $k_2\in\mathbb R\setminus \{0\}$, $c_2\in\mathbb R\setminus I_3$.
\end{theorem}
\begin{proof}
We have: $a_{12}=f_2\frac{\displaystyle f_1'}{\displaystyle f_1}$, $a_{23}=f_3\frac{\displaystyle f_2'}{\displaystyle f_2}$, $a_{31}=f_1\frac{\displaystyle f_3'}{\displaystyle f_3}$, and $a_{13}=a_{21}=a_{32}=\nolinebreak 0$, and we get:
\begin{align*}
\R(E_1,E_1)&=E_1(a_{31})+E_2(a_{12})-a_{12}^2-a_{31}^2,\\
\R(E_2,E_2)&=E_2(a_{12})+E_3(a_{23})-a_{12}^2-a_{23}^2,\\
\R(E_3,E_3)&=E_1(a_{31})+E_3(a_{23})-a_{31}^2-a_{23}^2,\\
\R(E_1,E_2)&=a_{31}a_{12},\\
\R(E_1,E_3)&=E_1(a_{23}),\\
\R(E_2,E_3)&=a_{12}a_{23}.
\end{align*}

Then, $\R=0$ if and only if
\begin{equation*}
\left\{
    \begin{aligned}
&f_1^2\Big(\frac{\displaystyle f_3'}{\displaystyle f_3}\Big)'-\Big(f_1\frac{\displaystyle f_3'}{\displaystyle f_3}\Big)^2+ f_2^2\Big(\frac{\displaystyle f_1'}{\displaystyle f_1}\Big)'-\Big(f_2\frac{\displaystyle f_1'}{\displaystyle f_1}\Big)^2=0\\
&f_2^2\Big(\frac{\displaystyle f_1'}{\displaystyle f_1}\Big)'-\Big(f_2\frac{\displaystyle f_1'}{\displaystyle f_1}\Big)^2+ f_3^2\Big(\frac{\displaystyle f_2'}{\displaystyle f_2}\Big)'-\Big(f_3\frac{\displaystyle f_2'}{\displaystyle f_2}\Big)^2=0\\
&f_1^2\Big(\frac{\displaystyle f_3'}{\displaystyle f_3}\Big)'-\Big(f_1\frac{\displaystyle f_3'}{\displaystyle f_3}\Big)^2+ f_3^2\Big(\frac{\displaystyle f_2'}{\displaystyle f_2}\Big)'-\Big(f_3\frac{\displaystyle f_2'}{\displaystyle f_2}\Big)^2=0\\
&f_1'f_2'=f_2'f_3'=f_3'f_1'=0
    \end{aligned}
  \right. ,
\end{equation*}
which gives
\begin{equation}\label{eqT2.5}
\left\{
    \begin{aligned}
&f_1^2(h_3'-h_3^2)=- f_2^2(h_1'-h_1^2)\\
&f_2^2(h_1'-h_1^2)=- f_3^2(h_2'-h_2^2)\\
&f_1^2(h_3'-h_3^2)=- f_3^2(h_2'-h_2^2)\\
&f_1'f_2'=f_2'f_3'=f_3'f_1'=0
    \end{aligned}
  \right. .
\end{equation}

From the first three equations, we deduce that $h_i'=h_i^2$ for any $i\in\nolinebreak\{1,2,3\}$, and, from the last three equations, we infer that two of the functions $f_1$, $f_2$, and $f_3$ must be constant.

For $f_1$, $f_2$, and $f_3$ constant on $I$, the system is verified.

If, for example, we have $f_3'\neq 0$ at every point of a maximal open interval $J_1\subseteq I_1$, then $h_3\neq 0$ at every point of $J_1$, and we infer that $\frac{\displaystyle h_3'}{\displaystyle h_3^2}=1$ on $J_1$, which, by integration, gives
$$\frac{f_3'(x^1)}{f_3(x^1)}=h_3(x^1)=\frac{-1}{x^1-c_3},$$
where $c_3\in\mathbb R\setminus J_1$. It follows that $h_3$ has the above expression on $I_1$. Then,
$$f_3(x^1)=\frac{k_3}{x^1-c_3}$$
on $I_1$, with $k_3\in\mathbb R\setminus \{0\}$, $c_3\in\mathbb R\setminus I_1$.
Hence, $f_3'(x^1)\neq 0$ for any $x^1\in I_1$, so $f_1'=0$ on $I_2$, and $f_2'=0$ on $I_3$, i.e., $f_1$ and $f_2$ are constant functions.

Due to the circular symmetry of (\ref{eqT2.5}), we get the statement of the theorem.
\end{proof}

\begin{theorem}\label{ps11n1}
If $f_1=f_1(x^1)$, $f_2=f_2(x^1)$, $f_3=f_3(x^2)$, then $(I,g)$ is a flat Riemannian manifold if and only if one of the following assertions holds:

(1) $f_2=k_2\in \mathbb R\setminus\{0\}$ and $f_3=k_3\in \mathbb R\setminus\{0\}$;

(2) $f_2=k_2\in \mathbb R\setminus\{0\}$ and $f_3(x^2)=\frac{\displaystyle k_3}{\displaystyle x^2-c_3}$, with $k_3\in \mathbb R\setminus\{0\}$, $c_3\in \mathbb R\setminus I_2$;

(3) $f_2=\frac{\displaystyle k_2}{\displaystyle F-c_2}$ and $f_3=k_3\in \mathbb R\setminus \{0\}$, where $F$ is an antiderivative of $\frac{\displaystyle 1}{\displaystyle f_1}$, $k_2\in \mathbb R \setminus \{0\}$, and $c_2\in \mathbb R\setminus \{F(x^1)|\ x^1\in I_1\}$.
\end{theorem}
\begin{proof}
We have: $a_{21}=f_1\frac{\displaystyle f_2'}{\displaystyle f_2}$, $a_{32}=f_2\frac{\displaystyle f_3'}{\displaystyle f_3}$, and $a_{12}=a_{13}=a_{23}=a_{31}=0$, and we get:
\begin{align*}
\R(E_1,E_1)&=E_1(a_{21})-a_{21}^2,\\
\R(E_2,E_2)&=E_1(a_{21})+E_2(a_{32})-a_{21}^2-a_{32}^2,\\
\R(E_3,E_3)&=E_2(a_{32})-a_{32}^2,\\
\R(E_1,E_2)&=E_1(a_{32}),\\
\R(E_1,E_3)&=\R(E_2,E_3)=0.
\end{align*}

Then, $\R=0$ if and only if
\begin{equation*}
\left\{
    \begin{aligned}
&      f_1'\frac{\displaystyle f_2'}{\displaystyle f_2}+f_1\Big[\Big(\frac{\displaystyle f_2'}{\displaystyle f_2}\Big)'-\Big(\frac{\displaystyle f_2'}{\displaystyle f_2}\Big)^2\Big]=0\\
&\Big(\frac{\displaystyle f_3'}{\displaystyle f_3}\Big)'=\Big(\frac{\displaystyle f_3'}{\displaystyle f_3}\Big)^2\\
&f_2'f_3'=0
    \end{aligned}
  \right. ,
\end{equation*}
which gives
\begin{equation*}
\left\{
    \begin{aligned}
&h_1h_2+h_2'-h_2^2=0\\
&h_3'=h_3^2\\
&h_2 h_3=0
    \end{aligned}
  \right. .
\end{equation*}

For $h_2=0$ and $h_3=0$, i.e., $f_2$ and $f_3$ constant on $I$, the system is verified.

If we have $h_3\neq 0$, that is, $h_3(x^2)\neq 0$ at every point of a maximal open interval $J_2\subseteq I_2$, then $\frac{\displaystyle h_3'}{\displaystyle h_3^2}=1$ on $J_2$, which, by integration, gives
$$\frac{f_3'(x^2)}{f_3(x^2)}=h_3(x^2)=\frac{-1}{x^2-c_3},$$
where $c_3\in\mathbb R\setminus J_2$. Due to continuity, $h_3(x^2)$ has the above expression on $I_2$. So,
$$f_3(x^2)=\frac{k_3}{x^2-c_3}\ \textrm{ and }h_3(x^2)\neq 0$$
everywhere on $I_2$, with $k_3\in\mathbb R\setminus \{0\}$, $c_3\in\mathbb R\setminus I_2$. It follows that $h_2=0$, that is, $f_2=k_2\in \mathbb R\setminus \{0\}$ on $I$.

If we have $h_3=0$, i.e., $f_3=k_3\in \mathbb R\setminus \{0\}$ on $I$, and $h_2\neq 0$, then, since $h_2'=\nolinebreak h_2(h_2-h_1)$, with the same proof as in Theorem \ref{ps1b}, we infer that $h_2(x^1)\neq 0$ for any $x^1\in I_1$, and $f_2= \displaystyle\frac{k_2}{F-c_2}$, where $F=F(x^1)$ is an antiderivative of $\frac{\displaystyle 1}{\displaystyle f_1}$, $k_2\in \mathbb R \setminus \{0\}$, and $c_2\in \mathbb R$ such that $F(x^1)\neq c_2$ for any $x^1\in I_1$.
\end{proof}

\begin{theorem}\label{ps11n2}
If $f_1=f_1(x^1)$, $f_2=f_2(x^1)$, $f_3=f_3(x^3)$, then $(I,g)$ is a flat Riemannian manifold if and only if one of the following assertions holds:

(1) $f_2=k_2\in \mathbb R\setminus\{0\}$;

(2) $f_2= \displaystyle\frac{k_2}{F-c_2}$, where $F$ is an antiderivative of $\frac{\displaystyle 1}{\displaystyle f_1}$, $k_2\in \mathbb R \setminus \{0\}$, and $c_2\in \mathbb R$ such that $F(x^1)\neq c_2$ for any $x^1\in I_1$.
\end{theorem}
\begin{proof}
We have
$$a_{21}=f_1\frac{\displaystyle f_2'}{\displaystyle f_2}$$ and $$a_{ij}=0 \ \ \textrm{for} \ \ (i,j)\in\{(1,2),(1,3),(2,3),(3,1),(3,2)\},$$ and we get:
\begin{align*}
\R(E_1,E_1)&=\R(E_2,E_2)=E_1(a_{21})-a_{21}^2,\\
\R(E_3,E_3)&=\R(E_1,E_2)=\R(E_1,E_3)=\R(E_2,E_3)=0.
\end{align*}

Then, $\R=0$ if and only if
\begin{equation*}
f_1'\frac{\displaystyle f_2'}{\displaystyle f_2}+f_1\Big[\Big(\frac{\displaystyle f_2'}{\displaystyle f_2}\Big)'-\Big(\frac{\displaystyle f_2'}{\displaystyle f_2}\Big)^2\Big]=0,
    \end{equation*}
which gives
\begin{equation*}
h_1h_2+h_2'-h_2^2=0.
\end{equation*}

With the same proof as in Theorem \ref{ps1b}, we infer that $f_2$ is constant on $I$, or $f_2=\nolinebreak \displaystyle\frac{k_2}{F-c_2}$, where $F=F(x^1)$ is an antiderivative of $\frac{\displaystyle 1}{\displaystyle f_1}$, $k_2\in \mathbb R \setminus \{0\}$, and $c_2\in \mathbb R$ such that $F(x^1)\neq c_2$ for any $x^1\in I_1$.
\end{proof}

\begin{theorem}\label{ps11n3}
If $f_1=f_1(x^3)$, $f_2=f_2(x^3)$, $f_3=f_3(x^1)$, then $(I,g)$ is a flat Riemannian manifold if and only if one of the following assertions holds:

(1) $f_1=k_1\in \mathbb R\setminus\{0\}$, $f_2=k_2\in \mathbb R\setminus\{0\}$, and $f_3=k_3\in \mathbb R\setminus\{0\}$;

(2) $f_1(x^3)=\frac{\displaystyle k_1}{\displaystyle x^3-c_1}$, $f_2=k_2\in \mathbb R\setminus\{0\}$, and $f_3=k_3\in \mathbb R\setminus\{0\}$, with $k_1\in \mathbb R\setminus\{0\}$, $c_1\in \mathbb R\setminus I_3$;

(3) $f_1=k_1\in \mathbb R\setminus\{0\}$, $f_2(x^3)=\frac{\displaystyle k_2}{\displaystyle x^3-c_2}$, and $f_3=k_3\in \mathbb R\setminus\{0\}$, with $k_2\in \mathbb R\setminus\{0\}$, $c_2\in \mathbb R\setminus I_3$;

(4) $f_1=k_1\in \mathbb R\setminus\{0\}$, $f_2=k_2\in \mathbb R\setminus\{0\}$, and $f_3(x^1)=\frac{\displaystyle k_3}{\displaystyle x^1-c_3}$, with $k_3\in \mathbb R\setminus\{0\}$, $c_3\in \mathbb R\setminus I_1$;

(5) $f_1(x^3)=\frac{\displaystyle k_1}{\displaystyle x^3-c_1}$, $f_2=k_2\in \mathbb R\setminus\{0\}$, and $f_3(x^1)=\frac{\displaystyle k_3}{\displaystyle x^1-c_3}$, with $k_1, k_3\in \mathbb R\setminus\{0\}$, $c_1\in \mathbb R\setminus I_3$, $c_3\in \mathbb R\setminus I_1$;

(6) $f_1(x^3)=\frac{\displaystyle 1}{\displaystyle F_1^{-1}(\eta_1 x^3+d_1)}$ or \\
$f_1(x^3)=\left\{
\begin{aligned}
&\frac{\displaystyle 1}{\displaystyle F_1^{-1}(-\varepsilon_1 x^3+2\varepsilon_1x_0^3+c_1)} &\textrm{ for }x^3<x_0^3\\
&\frac{\displaystyle 1}{\displaystyle \lim_{t^3\searrow x_0^3}F_1^{-1}(\varepsilon_1t^3+c_1)} &\textrm{ for }x^3=x_0^3\\
&\frac{\displaystyle 1}{\displaystyle F_1^{-1}(\varepsilon_1 x^3+c_1)} &\textrm{ for }x^3>x_0^3
\end{aligned}
\right. ,$
$f_3(x^1)=\frac{\displaystyle 1}{\displaystyle F_3^{-1}(\eta_3 x^1+d_3)}$ or \\
$f_3(x^1)=\left\{
\begin{aligned}
&\frac{\displaystyle 1}{\displaystyle F_3^{-1}(-\varepsilon_3 x^1+2\varepsilon_3x_0^1+c_3)} &\textrm{ for }x^1<x_0^1\\
&\frac{\displaystyle 1}{\displaystyle \lim_{t^1\searrow x_0^1}F_3^{-1}(\varepsilon_3t^1+c_3)} &\textrm{ for }x^1=x_0^1\\
&\frac{\displaystyle 1}{\displaystyle F_3^{-1}(\varepsilon_3 x^1+c_3)} &\textrm{ for }x^1>x_0^1
\end{aligned}
\right. ,$
and $f_2=k_2\in \mathbb R\setminus\{0\}$, where $F_i$ is an antiderivative of the function
$$z\mapsto \frac{1}{\sqrt{-2k_i\ln \vert z\vert +r_i}}$$
on a maximal connected domain $D_i\subseteq \mathbb R\setminus \{0\}$, with $k_3=-k_1\in\nolinebreak \mathbb R\setminus\nolinebreak \{0\}$, $r_i\in \mathbb R$, $i\in \{1,3\}$, and $\varepsilon_i:=\sign (-k_iz_0^i)$, $x_0^j\in\nolinebreak I_j$, $c_i=\lim_{z\rightarrow z_0^i}F_i(z)-\varepsilon_i x_0^j$ for $z_0^i\in \partial D_i\setminus \{0\}$, and $d_i\in \mathbb R$, $\eta_i\in \{\pm1\}$ such that $\eta_i x^j+d_i\in F_i(D_i)$ for any $x^j\in I_j$, $(i,j)\in\nolinebreak \{(1,3),(3,1)\}$.
\end{theorem}
\begin{proof}
We have: $a_{13}=f_3\frac{\displaystyle f_1'}{\displaystyle f_1}$, $a_{23}=f_3\frac{\displaystyle f_2'}{\displaystyle f_2}$, $a_{31}=f_1\frac{\displaystyle f_3'}{\displaystyle f_3}$, and $a_{12}=a_{21}=a_{32}=\nolinebreak 0$, and we get:
\begin{align*}
\R(E_1,E_1)&=E_1(a_{31})-a_{31}^2+E_3(a_{13})-a_{13}^2-a_{13}a_{23},\\
\R(E_2,E_2)&=E_3(a_{23})-a_{23}^2-a_{13}a_{23},\\
\R(E_3,E_3)&=E_1(a_{31})-a_{31}^2+E_3(a_{13})-a_{13}^2+E_3(a_{23})-a_{23}^2,\\
\R(E_1,E_3)&=E_1(a_{23}),\\
\R(E_2,E_3)&=E_2(a_{13}),\\
\R(E_1,E_2)&=0.
\end{align*}

Then, $\R=0$ if and only if
\begin{equation*}
\left\{
    \begin{aligned}
&      f_1^2\Big[\Big(\frac{\displaystyle f_3'}{\displaystyle f_3}\Big)'-\Big(\frac{\displaystyle f_3'}{\displaystyle f_3}\Big)^2\Big]+f_3^2\Big[\Big(\frac{\displaystyle f_1'}{\displaystyle f_1}\Big)'-\Big(\frac{\displaystyle f_1'}{\displaystyle f_1}\Big)^2\Big]=0\\
&\Big(\frac{\displaystyle f_2'}{\displaystyle f_2}\Big)'=\Big(\frac{\displaystyle f_2'}{\displaystyle f_2}\Big)^2\\
&f_1'f_2'=f_2'f_3'=0
    \end{aligned}
  \right. ,
\end{equation*}
which gives
\begin{equation*}
\left\{
    \begin{aligned}
&f_1^2(h_3'-h_3^2)+f_3^2(h_1'-h_1^2)=0\\
&h_2'=h_2^2\\
&h_1h_2=h_2h_3=0
    \end{aligned}
  \right. .
\end{equation*}

If we have $h_2\neq 0$, then $h_2(x^3)\neq 0$ at every point of a maximal open interval $J_3\subseteq I_3$, and it follows that $\frac{\displaystyle h_2'}{\displaystyle h_2^2}=1$ on $J_3$, which, by integration, gives
$$\frac{f_2'(x^3)}{f_2(x^3)}=h_2(x^3)=\frac{-1}{x^3-c_2},$$
where $c_2\in\mathbb R\setminus J_3$. We deduce that $J_3=I_3$, so $h_2$ has the above expression on $I_3$, and $h_1=h_3=0$ on $I$, i.e., $f_1=k_1\in \mathbb R\setminus\{0\}$, and $f_3=k_3\in \mathbb R\setminus\{0\}$, which satisfy the system. We get
$$f_2(x^3)=\frac{k_2}{x^3-c_2}$$
on $I$, with $k_2\in\mathbb R\setminus \{0\}$, $c_2\in\mathbb R\setminus I_3$.

Let now $h_2=0$, that is, $f_2=k_2\in \mathbb R\setminus \{0\}$ on $I$. Then, the last two equations of the system are verified. The first equation of the previous system is equivalent to
$$\frac{\displaystyle 1}{\displaystyle f_1^2}\Big[\Big(\frac{\displaystyle f_1'}{\displaystyle f_1}\Big)'-\Big(\frac{\displaystyle f_1'}{\displaystyle f_1}\Big)^2\Big]=-\frac{\displaystyle 1}{\displaystyle f_3^2}\Big[\Big(\frac{\displaystyle f_3'}{\displaystyle f_3}\Big)'-\Big(\frac{\displaystyle f_3'}{\displaystyle f_3}\Big)^2\Big].$$
Since the left term from the above equation depends only on $x^3$, and the right term depends only on $x^1$, we deduce that they must be constant; therefore, there exists $k\in \mathbb R$ such that
\begin{equation}\label{eq1T2.8}
\frac{\displaystyle f_1f_1''-2(f_1')^2}{\displaystyle f_1^4}=k,\textrm{ and }\
\frac{\displaystyle f_3f_3''-2(f_3')^2}{\displaystyle f_3^4}=-k.
\end{equation}

If $k=0$, from the last equation, we get $h_3'-h_3^2=0$, which has the solutions $h_3=0$, i.e., $f_3$ constant on $I$, and $h_3(x^1)=\frac {\displaystyle -1}{\displaystyle x^1-c_3}$, i.e.,
$f_3(x^1)=\frac {\displaystyle k_3}{\displaystyle x^1-c_3}$, for any $x^1\in I_1$, with $c_3\in \mathbb R\setminus I_1$, $k_3\in \mathbb R\setminus \{0\}$. From the first equation of (\ref{eq1T2.8}), we similarly infer that
$h_1=0$, i.e., $f_1$ constant on $I$, or $h_1(x^3)=\frac {\displaystyle -1}{\displaystyle x^3-c_1}$, i.e.,
$f_1(x^3)=\frac {\displaystyle k_1}{\displaystyle x^3-c_1}$, for any $x^3\in I_3$, with $c_1\in \mathbb R\setminus I_3$, $k_1\in \mathbb R\setminus \{0\}$.

Let now $k\in \mathbb R\setminus\{0\}$.
Defining $g:=\frac {\displaystyle 1}{\displaystyle f_1}$, the first equation of (\ref{eq1T2.8}) becomes $g''g=-k$. It follows that $g''(x^3)\neq 0$ for any $x^3\in I_3$, so $g'$ is strictly monotone on $I_3$.

Suppose that there exists $x_0^3\in I_3$ such that $f_1'(x_0^3)= 0$, which is equivalent to $g'(x_0^3)=\nolinebreak 0$. Then, $g'(x^3)$ has a different sign to the left than to the right of $x_0^3$.

Let us denote $I_3=:(a^3,b^3)$, where $a^3,b^3\in \overline{\mathbb R}$, and consider, e.g., that $g'<0$ to the left and $g'>0$ to the right of $x_0^3$. This happens when $g''>0$, that is, $k>0$ and $g<0$, or $k<0$ and $g>0$. Then, $g$ is strictly decreasing on $J_1:=(a^3,x_0^3)$ and strictly increasing on $J_2:=(x_0^3,b^3)$. We will define $u_i:g(J_i)\rightarrow \mathbb R$, $u_i(g(x^3)):=g'(x^3)$ for any $x^3\in J_i$, $i\in \{1,2\}$. Since $u_i(z)=(g'\circ g|_{L_i}^{-1})(z)$ for $z\in g(J_i)$, $u_i$ is continously differentiable, $i\in \{1,2\}$. We infer that  $u_i'(g(x^3))g'(x^3)=g''(x^3)$; hence,
$u_i'(g(x^3))u_i(g(x^3))=\frac{\displaystyle -k}{\displaystyle g(x^3)}$ for $x^3\in J_i$, and we deduce that $u_i'(z)u_i(z)=\frac{\displaystyle -k}{\displaystyle z}$ for any $z\in L_i:=g(J_i)\subset \mathbb R\setminus \{0\}$, $i\in \{1,2\}$. It follows that $u_i^2(z)=-2k\ln \vert z\vert +r_i$ on $L_i$, with $r_i\in \mathbb R$ such that $-2k\ln \vert z\vert +r_i>0$ for any $z\in L_i$, $i\in \{1,2\}$. Because
$$g'(x_0^3)=0=\lim_{x^3\rightarrow x_0^3}g'(x^3)=\lim_{x^3\nearrow x_0^3}u_1(g(x^3))=\lim_{x^3\searrow x_0^3}u_2(g(x^3)),$$
we infer that $r_1=r_2=2k\ln \vert g(x_0^3)\vert$. Since $g'<0$ to the left, and $g'>0$ to the right of $x_0^3$, we have $u_1<0$ and $u_2>0$, so
$$u_1(g(x^3))=-\sqrt{-2k\ln \vert g(x^3)\vert +r_1},\ \textrm{ and }u_2(g(x^3))=\sqrt{-2k\ln \vert g(x^3)\vert +r_1},$$
and we get:
\begin{align*}
\frac{\displaystyle g'(x^3)}{\displaystyle \sqrt{-2k\ln \vert g(x^3)\vert +r_1}}&=-1 \textrm{ for }x^3<x_0^3,\\
\frac{\displaystyle g'(x^3)}{\displaystyle \sqrt{-2k\ln \vert g(x^3)\vert +r_1}}&=1 \hspace{10pt} \textrm{ for }x^3>x_0^3.
\end{align*}
Denote $g(x_0^3)=:z_0$ and let $F_1$ be an antiderivative of the function
$$z\mapsto \frac{1}{\sqrt{-2k_1\ln \vert z\vert +r_1}}$$
on the maximal connected domain $D_1\subseteq \mathbb R\setminus \{0\}$ such that $z_0\in \partial D_1$, where $k_1:=k$, $r_1:=2k_1\ln \vert z_0\vert$. Then,
$$(F_1\circ g)'(x^3)=\left\{
\begin{aligned}
-1\ \textrm{ for }x^3\in (a^3,x_0^3)\\
\ 1\ \textrm{ for }x^3\in (x_0^3,b^3)
\end{aligned}
\right. ,
$$
thus
$$F_1 (g(x^3))=\left\{
\begin{aligned}
-x^3+c_2\ \textrm{ for }x^3\in (a^3,x_0^3)\\
x^3+c_1\ \textrm{ for }x^3\in (x_0^3,b^3)
\end{aligned}
\right. .
$$
Since
$$-x_0^3+c_2=\lim_{x^3\nearrow x_0^3}F_1(g(x^3))=\lim_{z\searrow z_0}F_1(z)= \lim_{x^3\searrow x_0^3}F_1(g(x^3))=x_0^3+c_1,$$
we get
$$c_1=\lim_{z\searrow z_0}F_1(z)-x_0^3,\textrm{ and } c_2=2x_0^3+c_1=\lim_{z\searrow z_0}F_1(z)+x_0^3.$$
Thus,
$$g(x_0^3)=z_0=\lim_{x^3\searrow x_0^3}F_1^{-1}(x^3+c_1),$$
and
$$F_1 (g(x^3))=\left\{
\begin{aligned}
&-x^3+2x_0^3+c_1\ \textrm{ for }x^3\in (a^3,x_0^3)\\
&\ \ x^3+c_1\ \hspace{40pt}\textrm{ for }x^3\in (x_0^3,b^3)
\end{aligned}
\right. ;
$$
hence,
\begin{equation*}
g(x^3)=\left\{
\begin{aligned}
&F_1^{-1}(-x^3+2x_0^3+c_1)&&\ \textrm{ for }x^3\in (a^3,x_0^3)\\
&\lim_{t^3\searrow x_0^3}F_1^{-1}(t^3+c_1)&&\ \textrm{ for }x^3=x_0^3\\
&F_1^{-1}(x^3+c_1)&&\ \textrm{ for }x^3\in (x_0^3,b^3)
\end{aligned}
\right. .
\end{equation*}

Consider now that $g'>0$ to the left and $g'<0$ to the right of $x_0^3$. This happens when $g''<0$, that is, $k>0$ and $g>0$, or $k<0$ and $g<0$. Using a similar argument as above, with the same definitions of $z_0$, $k_1$, $r_1$, and $F_1$, we get
$$F_1 (g(x^3))=\left\{
\begin{aligned}
x^3+c_2\ \textrm{ for }x^3\in (a^3,x_0^3)\\
-x^3+c_1\ \textrm{ for }x^3\in (x_0^3,b^3)
\end{aligned}
\right. ,
$$
where
$$c_1=\lim_{z\nearrow z_0}F_1(z)+x_0^3,\textrm{ and } c_2=-2x_0^3+c_1=\lim_{z\nearrow z_0}F_1(z)-x_0^3;$$
hence,
\begin{equation*}
g(x^3)=\left\{
\begin{aligned}
&F_1^{-1}(x^3-2x_0^3+c_1)&&\ \textrm{ for }x^3\in (a^3,x_0^3)\\
&\lim_{t^3\searrow x_0^3}F_1^{-1}(-t^3+c_1)&&\ \textrm{ for }x^3=x_0^3\\
&F_1^{-1}(-x^3+c_1)&&\ \textrm{ for }x^3\in (x_0^3,b^3)
\end{aligned}
\right. .
\end{equation*}

The two cases considered above can be expressed in a single one:
for any sign of $g''$, with the same definitions of $z_0$, $k_1$, $r_1$, and $F_1$, denoting $\varepsilon_1:=\sign (-k_1z_0)$, we get
$$F_1 (g(x^3))=\left\{
\begin{aligned}
-\varepsilon_1 x^3+c_2\ \textrm{ for }x^3\in (a^3,x_0^3)\\
\varepsilon_1 x^3+c_1\ \textrm{ for }x^3\in (x_0^3,b^3)
\end{aligned}
\right. ,
$$
where
$$c_1=\lim_{z\rightarrow z_0}F_1(z)-\varepsilon_1 x_0^3\ \textrm{ and }
c_2=2\varepsilon_1x_0^3+c_1=\lim_{z\rightarrow z_0}F_1(z)+\varepsilon_1 x_0^3;$$
hence,
\begin{equation}\label{eq2T2.8}
g(x^3)=\left\{
\begin{aligned}
&F_1^{-1}(-\varepsilon_1 x^3+2\varepsilon_1 x_0^3+c_1)&&\ \textrm{ for }x^3\in (a^3,x_0^3)\\
&\lim_{t^3\searrow x_0^3}F_1^{-1}(\varepsilon_1 t^3+c_1)&&\ \textrm{ for }x^3=x_0^3\\
&F_1^{-1}(\varepsilon_1 x^3+c_1)&&\ \textrm{ for }x^3\in (x_0^3,b^3)
\end{aligned}
\right. ,
\end{equation}
and
\begin{equation}\label{eq3T2.8}
f_1(x^3)=\left\{
\begin{aligned}
&\frac{\displaystyle 1}{\displaystyle F_1^{-1}(-\varepsilon_1 x^3+2\varepsilon_1x_0^3+c_1)} &\textrm{ for }x^3<x_0^3\\
&\frac{\displaystyle 1}{\displaystyle \lim_{t^3\searrow x_0^3}F_1^{-1}(\varepsilon_1t^3+c_1)} &\textrm{ for }x^3=x_0^3\\
&\frac{\displaystyle 1}{\displaystyle F_1^{-1}(\varepsilon_1 x^3+c_1)} &\textrm{ for }x^3>x_0^3
\end{aligned}
\right. .
\end{equation}

Conversely, for $k\in \mathbb R\setminus\{0\}$, let $F_1$ be an antiderivative of the function
$$z\mapsto \frac{1}{\sqrt{-2k_1\ln \vert z\vert +r_1}}$$
on a maximal connected domain $D_1\subseteq \mathbb R\setminus \{0\}$, where $r_1\in \mathbb R$, $k_1:=\nolinebreak k$. Let $z_0\in \partial D_1\setminus \{0\}$, so we have $r_1=2k_1\ln \vert z_0\vert$, and let $x_0^3\in I_3=:(a^3,b^3)$, $\varepsilon_1:=\nolinebreak \sign (-k_1z_0)$, and $f_1:\nolinebreak I_3\rightarrow \mathbb R$ be defined by (\ref{eq3T2.8}), where $c_1:=\lim_{z\rightarrow z_0}F_1(z)-\varepsilon_1 x_0^3$. The function $g=\frac {\displaystyle 1}{\displaystyle f_1}:\nolinebreak I_3\rightarrow \mathbb R$ satisfies (\ref{eq2T2.8}), is continuous, nowhere zero, verifies $g(x_0^3)=z_0$, and
$$g'(x^3)=\left\{
\begin{aligned}
&-\varepsilon_1\sqrt{-2k_1\ln \vert F_1^{-1}(-\varepsilon_1 x^3+2\varepsilon_1 x_0^3+c_1)\vert +r_1}
&\ \textrm{ for }x^3\in (a^3,x_0^3)\\
&\ \varepsilon_1\sqrt{-2k_1\ln \vert F_1^{-1}(\varepsilon_1 x^3+c_1)\vert +r_1}
&\ \textrm{ for }x^3\in (x_0^3,b^3)
\end{aligned}
\right. .
$$
Because \\
$\lim_{x^3\nearrow x_0^3}\sqrt{-2k_1\ln \vert F_1^{-1}(-\varepsilon_1 x^3+2\varepsilon_1 x_0^3+c_1)\vert +r_1}=$
\begin{align*}
&=\lim_{x^3\searrow x_0^3}\sqrt{-2k_1\ln \vert F_1^{-1}(\varepsilon_1 x^3+c_1)\vert +r_1}\\
&=\sqrt{-2k_1\ln \vert g(x_0^3)\vert +r_1}=0,
\end{align*}
we infer, through Lagrange's Theorem, that there exists
$g'(x_0^3)=0$.

It follows that $g'$ is continuous, and
$$g'(x^3)=\left\{
\begin{aligned}
&-\varepsilon_1\sqrt{-2k_1\ln \vert g(x^3)\vert +r_1}
&&\ \textrm{ for }x^3\in (a^3,x_0^3)\\
&\ 0 &&\ \textrm{ for }x^3=x_0^3\\
&\ \varepsilon_1\sqrt{-2k_1\ln \vert g(x^3)\vert +r_1}
&&\ \textrm{ for }x^3\in (x_0^3,b^3)
\end{aligned}
\right. ,
$$
from which we get
$$g''(x^3)=\left\{
\begin{aligned}
&-\varepsilon_1\frac{\displaystyle -k_1g'(x^3)}{\displaystyle g(x^3)\sqrt{-2k_1\ln \vert g(x^3)\vert +r_1}} &&\ \textrm{ for }x^3\in (a^3,x_0^3)\\
&\ \varepsilon_1\frac{\displaystyle -k_1g'(x^3)}{\displaystyle g(x^3)\sqrt{-2k_1\ln \vert g(x^3)\vert +r_1}} &&\ \textrm{ for }x^3\in (x_0^3,b^3)
\end{aligned}
\right. ,
$$
that is,
$$g''(x^3)=\frac{\displaystyle -k_1}{\displaystyle g(x^3)}\ \textrm{ for }x^3\in I_3\setminus\{x_0^3\},$$
which, through Lagrange's Theorem, leads to the existence of
$$g''(x_0^3)=\lim_{x^3\rightarrow x_0^3}\frac{\displaystyle -k_1}{\displaystyle g(x^3)}=
\frac{\displaystyle -k_1}{\displaystyle g(x_0^3)}.$$
So,
$$g''(x^3)=\frac{\displaystyle -k}{\displaystyle g(x^3)}\ \textrm{ for any }x^3\in I_3, $$
from which it also follows that $g$ is a smooth function on $I_3$; hence, $f_1=\frac {\displaystyle 1}{\displaystyle g}$ is smooth on $I_3$ and satisfies $f_1'(x_0^3)=0$ and $\frac{\displaystyle f_1f_1''-2(f_1')^2}{\displaystyle f_1^4}=k$.

The second equation of (\ref{eq1T2.8}),
\begin{equation*}\label{eq4T2.8}
\frac{\displaystyle f_3f_3''-2(f_3')^2}{\displaystyle f_3^4}=-k,
\end{equation*}
for $k\in \mathbb R\setminus\{0\}$, is equivalent to $g''g=k$, where $g=\frac {\displaystyle 1}{\displaystyle f_3}$. Using a proof similar to the one above, we deduce that, for any $x_0^1\in I_1=:(a^1,b^1)$, $a^1,b^1\in \overline{\mathbb R}$, a function $g:I_1\rightarrow\mathbb R$ is smooth and satisfies $g''g=k$ and $g'(x_0^1)=0$, i.e., $f_3:=\frac {\displaystyle 1}{\displaystyle g}$ is smooth and satisfies the second equation of (\ref{eq1T2.8}), with $f_3'(x_0^1)=0$, if and only if, for $F_3$ an antiderivative of the function
$$z\mapsto \frac{1}{\sqrt{-2k_3\ln \vert z\vert +r_3}}$$
on a maximal connected domain $D_3\subseteq \mathbb R\setminus \{0\}$, where $r_3\in \mathbb R$, $k_3:=-k$, and for $z_0\in\nolinebreak \partial D_1\setminus\nolinebreak \{0\}$ and $\varepsilon_3:=\sign (-k_3z_0)$, so $r_3=2k_3\ln \vert z_0\vert$, we have
\begin{equation*}
g(x^1)=\left\{
\begin{aligned}
&F_3^{-1}(-\varepsilon_3 x^1+2\varepsilon_3 x_0^1+c_3)&&\ \textrm{ for }x^1\in (a^1,x_0^1)\\
&\lim_{t^1\searrow x_0^1}F_3^{-1}(\varepsilon_3 t^1+c_3)&&\ \textrm{ for }x^1=x_0^1\\
&F_3^{-1}(\varepsilon_3 x^1+c_3)&&\ \textrm{ for }x^1\in (x_0^1,b^1)
\end{aligned}
\right. ,
\end{equation*}
where $c_3=\lim_{z\rightarrow z_0}F_3(z)-\varepsilon_3 x_0^1$, that is,
$$f_3(x^1)=\left\{
\begin{aligned}
&\frac{\displaystyle 1}{\displaystyle F_3^{-1}(-\varepsilon_3 x^1+2\varepsilon_3x_0^1+c_3)} &\textrm{ for }x^1<x_0^1\\
&\frac{\displaystyle 1}{\displaystyle \lim_{t^1\searrow x_0^1}F_3^{-1}(\varepsilon_3t^1+c_3)} &\textrm{ for }x^1=x_0^1\\
&\frac{\displaystyle 1}{\displaystyle F_3^{-1}(\varepsilon_3 x^1+c_3)} &\textrm{ for }x^1>x_0^1
\end{aligned}
\right. .$$

In view of the idea of the demonstrations above, we will now characterize the solutions of (\ref{eq1T2.8}) which have no critical points. We notice that a function $f_i:\nolinebreak I_j\rightarrow\nolinebreak \mathbb R$ is a solution for one of the equations of (\ref{eq1T2.8}) and has no critical points, i.e., $f_i'(x^j)\neq 0$ for any $x^j\in I_j$, $(i,j)=(1,3)$ or $(i,j)=(3,1)$, if and only if $g_i''g_i=-k_i$ and $g_i'(x^j)\neq 0$ everywhere on $I_j$, where $g_i:=\frac {\displaystyle 1}{\displaystyle f_i}$, $k_1=k$, and $k_3=-k$, which is equivalent to
$$u_i'(g_i(x^j))u_i(g_i(x^j))=\frac{\displaystyle -k_i}{\displaystyle g_i(x^j)}$$
for $x^j\in I_j$, where $u_i:g_i(I_j)\rightarrow \mathbb R$, $u_i(g_i(x^j)):=g_i'(x^j)$, that is,
$$u_i^2(z)=-2k_i\ln \vert z\vert +r_i$$
for $z\in L_j:=g_i(I_j)$, with $r_i>2k_i\ln \vert z\vert$ for any $z\in L_j$, i.e.,
$$u_i(g_i(x^j))=\sqrt{-2k_i\ln \vert g_i(x^j)\vert +r_i},\ \textrm{ or }u_i(g_i(x^j))=-\sqrt{-2k_i\ln \vert g_i(x^j)\vert +r_i}.$$
For $F_i$ an antiderivative of the function
$$z\mapsto \frac{1}{\sqrt{-2k_i\ln \vert z\vert +r_i}}$$
on a maximal connected domain $D_i\subseteq \mathbb R\setminus \{0\}$, where $k_3=-k_1\in\nolinebreak \mathbb R\setminus\nolinebreak \{0\}$, $r_i\in \mathbb R$, $i\in \{1,3\}$, the above assertion becomes
$$(F_i\circ g_i)'(x^j)=1\ \textrm{ on }I_j,\textrm{ or }\ (F_i\circ g_i)'(x^j)=-1\ \textrm{ on }I_j,$$
that is,
$$F_i(g_i(x^j))=x^j+d_i\ \textrm{ on }I_j,\textrm{ or }\ F_i(g_i(x^j))=-x^j+d_i\ \textrm{ on }I_j,$$
where $d_i\in \mathbb R$ such that $x^j+d_i\in F_i(D_i)$, or $-x^j+d_i\in F_i(D_i)$, respectively, for any $x^j\in I_j$, which is equivalent to
$$g_i(x^j)=F_i^{-1}(x^j+d_i)\ \textrm{ on }I_j,\textrm{ or }\ g_i(x^j)=F_i^{-1}(-x^j+d_i)\ \textrm{ on }I_j,$$
that is,
$$f_i(x^j)=\frac{\displaystyle 1}{\displaystyle F_i^{-1}(\eta_i x^j+d_i)}$$
on $I_j$, where $d_i\in \mathbb R$, $\eta_i\in \{\pm1\}$ such that $\eta_i x^j+d_i\in F_i(D_i)$ for any $x^j\in I_j$, which completes the proof.
\end{proof}

\begin{theorem}\label{ps11n4}
If $f_1=f_1(x^2)$, $f_2=f_2(x^1)$, $f_3=f_3(x^3)$, then $(I,g)$ is a flat Riemannian manifold if and only if one of the following assertions holds:

(1) $f_1=k_1\in \mathbb R\setminus\{0\}$, and $f_2=k_2\in \mathbb R\setminus\{0\}$;

(2) $f_1(x^2)=\frac{\displaystyle k_1}{\displaystyle x^2-c_1}$, and $f_2=k_2\in \mathbb R\setminus\{0\}$, with $k_1\in \mathbb R\setminus\{0\}$, $c_1\in \mathbb R\setminus I_2$;

(3) $f_1=k_1\in \mathbb R\setminus\{0\}$, and $f_2(x^1)=\frac{\displaystyle k_2}{\displaystyle x^1-c_2}$, with $k_2\in \mathbb R\setminus\{0\}$, $c_2\in \mathbb R\setminus I_1$;

(4) $f_1(x^2)=\frac{\displaystyle k_1}{\displaystyle x^2-c_1}$, and $f_2(x^1)=\frac{\displaystyle k_2}{\displaystyle x^1-c_2}$, with $k_1, k_2\in \mathbb R\setminus\{0\}$, $c_1\in\nolinebreak \mathbb R\setminus\nolinebreak I_2$, $c_2\in \mathbb R\setminus I_1$;

(5) $f_1(x^2)=\frac{\displaystyle 1}{\displaystyle F_1^{-1}(\eta_1 x^2+d_1)}$ or \\
$f_1(x^2)=\left\{
\begin{aligned}
&\frac{\displaystyle 1}{\displaystyle F_1^{-1}(-\varepsilon_1 x^2+2\varepsilon_1x_0^2+c_1)} &&\textrm{for }x^2<x_0^2\\
&\frac{\displaystyle 1}{\displaystyle \lim_{z^2\searrow x_0^2}F_1^{-1}(\varepsilon_1z^2+c_1)} &&\textrm{for }x^2=x_0^2\\
&\frac{\displaystyle 1}{\displaystyle F_1^{-1}(\varepsilon_1 x^2+c_1)} &&\textrm{for }x^2>x_0^2
\end{aligned}
\right. ,$
$f_2(x^1)=\frac{\displaystyle 1}{\displaystyle F_2^{-1}(\eta_2 x^1+d_2)}$ or \\
$f_2(x^1)=\left\{
\begin{aligned}
&\frac{\displaystyle 1}{\displaystyle F_2^{-1}(-\varepsilon_2 x^1+2\varepsilon_2x_0^1+c_2)} &\textrm{for }x^1<x_0^1\\
&\frac{\displaystyle 1}{\displaystyle \lim_{z^1\searrow x_0^1}F_2^{-1}(\varepsilon_2z^1+c_2)} &\textrm{for }x^1=x_0^1\\
&\frac{\displaystyle 1}{\displaystyle F_2^{-1}(\varepsilon_2 x^1+c_2)} &\textrm{for }x^1>x_0^1
\end{aligned}
\right. ,$
where $F_i$ is an antiderivative of the function
$$y\mapsto \frac{1}{\sqrt{-2k_i\ln \vert y\vert +r_i}}$$
on a maximal connected domain $D_i\subseteq \mathbb R\setminus \{0\}$, with $k_2=-k_1\in \mathbb R\setminus \{0\}$, $r_i\in \mathbb R$, $i\in \{1,2\}$, and $\varepsilon_i:=\sign (-k_iz_0^i)$, $x_0^j\in\nolinebreak I_j$,
$c_i=\lim_{z\rightarrow z_0^i}F_i(z)-\varepsilon_i x_0^j$ for $z_0^i\in \partial D_i\setminus \{0\}$, and $d_i\in \mathbb R$, $\eta_i\in\nolinebreak \{\pm 1\}$ such that $\eta_i x^j+d_i\in F_i(D_i)$ for any $x^j\in I_j$, $(i,j)\in\nolinebreak \{(1,2),(2,1)\}$.
\end{theorem}
\begin{proof}
We have: $a_{12}=f_2\frac{\displaystyle f_1'}{\displaystyle f_1}$, $a_{21}=f_1\frac{\displaystyle f_2'}{\displaystyle f_2}$, and $a_{13}=a_{23}=a_{31}=a_{32}=0$, and we get:
\begin{align*}
\R(E_1,E_1)&=\R(E_2,E_2)=E_1(a_{21})-a_{21}^2+E_2(a_{12})-a_{12}^2,\\
\R(E_3,E_3)&=\R(E_1,E_2)=\R(E_1,E_3)=\R(E_2,E_3)=0.
\end{align*}

Then, $\R=0$ if and only if
\begin{equation*}
f_1^2\Big[\Big(\frac{\displaystyle f_2'}{\displaystyle f_2}\Big)'-\Big(\frac{\displaystyle f_2'}{\displaystyle f_2}\Big)^2\Big]+f_2^2\Big[\Big(\frac{\displaystyle f_1'}{\displaystyle f_1}\Big)'-\Big(\frac{\displaystyle f_1'}{\displaystyle f_1}\Big)^2\Big]=0,
\end{equation*}
which is equivalent to
$$\frac{\displaystyle 1}{\displaystyle f_1^2}\Big[\Big(\frac{\displaystyle f_1'}{\displaystyle f_1}\Big)'-\Big(\frac{\displaystyle f_1'}{\displaystyle f_1}\Big)^2\Big]=-\frac{\displaystyle 1}{\displaystyle f_2^2}\Big[\Big(\frac{\displaystyle f_2'}{\displaystyle f_2}\Big)'-\Big(\frac{\displaystyle f_2'}{\displaystyle f_2}\Big)^2\Big].$$
Since the left term from the above equation depends only on $x^2$, and the right term depends only on $x^1$, we deduce that they must be constant; therefore,
$$\frac{\displaystyle f_1f_1''-2(f_1')^2}{\displaystyle f_1^4}=-\frac{\displaystyle f_2f_2''-2(f_2')^2}{\displaystyle f_2^4}=k\in \mathbb R.$$
Further, with a proof similar to that of Theorem \ref{ps11n3}, we obtain the statement.
\end{proof}

\section{Applications to warped products}

\begin{theorem}\label{ps11mac8}
If $f_1=1$, $f_2=f_3=f(x^1)$, then the warped product manifold
$$I_1\times_{\frac{1}{f}}( I_2\times I_3)=(I,g)$$
is a flat Riemannian manifold if and only if $f$ is constant.

In this case, the manifold is just a direct product.
\end{theorem}
\begin{proof}
We have: $a_{21}=a_{31}=\frac{\displaystyle f'}{\displaystyle f}$, and $a_{12}=a_{13}=a_{23}=a_{32}=0$, and we get:
\begin{align*}
\R(E_1,E_1)&=2[E_1(a_{21})-a_{21}^2],\\
\R(E_2,E_2)&=E_1(a_{21})-a_{21}^2-a_{21}a_{31},\\
\R(E_3,E_3)&=E_1(a_{31})-a_{31}^2-a_{21}a_{31},\\
\R(E_1,E_2)&=\R(E_1,E_3)=\R(E_2,E_3)=0.
\end{align*}

Then, $\R=0$ if and only if
\begin{equation*}
\left\{
    \begin{aligned}
&\left(\displaystyle \frac{f'}{f}\right)'=\left(\displaystyle \frac{f'}{f}\right)^2\\
&f'=0
     \end{aligned}
  \right.\ ,
\end{equation*}
which is equivalent to $f$ constant.
\end{proof}

And we can further deduce
\begin{corollary}\label{corT3.8bd}
There do not exist proper flat warped product manifolds of the form
$$I_1\times_{f} ( I_2\times I_3)=\Big(I,\ g=(dx^1)^2+f^2[(dx^2)^2+(dx^3)^2]\Big).$$
\end{corollary}

\begin{theorem}\label{ps11mac7}
If $f_1=f_2=1$, $f_3=f_3(x^1,x^2)$, then the warped product manifold
$$(I_1\times I_2)\times_{\frac{1}{f_3}} I_3=(I,g)$$
is a flat Riemannian
manifold if and only if
$$f_3(x^1,x^2)=\frac{\displaystyle 1}{\displaystyle c_1x^1+c_2x^2+c_3},$$
with $c_1, c_2, c_3\in \mathbb R$ such that $c_1x^1+c_2x^2+c_3\neq 0$ for any $(x^1,x^2)\in I_1\times I_2$.

If $c_1=c_2=0$, then the manifold is just a direct product.
\end{theorem}
\begin{proof}
We have: $a_{31}=\frac{\displaystyle 1}{\displaystyle f_3}\cdot\frac{\displaystyle \partial f_3}{\displaystyle \partial x^1}$, $a_{32}=\frac{\displaystyle 1}{\displaystyle f_3}\cdot\frac{\displaystyle \partial f_3}{\displaystyle \partial x^2}$, and $a_{12}=a_{13}=a_{21}=a_{23}=0$, and we get:
\begin{align*}
\R(E_1,E_1)&=E_1(a_{31})-a_{31}^2,\\
\R(E_2,E_2)&=E_2(a_{32})-a_{32}^2,\\
\R(E_3,E_3)&=E_1(a_{31})+E_2(a_{32})-a_{31}^2-a_{32}^2,\\
\R(E_1,E_2)&=E_1(a_{32})-a_{31}a_{32},\\
\R(E_1,E_3)&=\R(E_2,E_3)=0.
\end{align*}

Then, $\R=0$ if and only if
\begin{equation}\label{T3.3}
\left\{
    \begin{aligned}
& \frac{\displaystyle \partial }{\displaystyle \partial x^1}\left(\frac{\displaystyle 1}{\displaystyle f_3}\cdot\frac{\displaystyle \partial f_3}{\displaystyle \partial x^1}\right)=\left(\frac{\displaystyle 1}{\displaystyle  f_3}\cdot\frac{\displaystyle \partial f_3}{\displaystyle \partial x^1}\right)^2\\
& \frac{\displaystyle \partial }{\displaystyle \partial x^2}\left(\frac{\displaystyle 1}{\displaystyle f_3}\cdot\frac{\displaystyle \partial f_3}{\displaystyle \partial x^2}\right)=\left(\frac{\displaystyle 1}{\displaystyle  f_3}\cdot\frac{\displaystyle \partial f_3}{\displaystyle \partial x^2}\right)^2\\
& \displaystyle \frac{\partial }{\partial x^1}\left(\displaystyle \frac{1}{f_3}\cdot\frac{\partial f_3}{\partial x^2}\right)=\left(\displaystyle \frac{1}{f_3}\cdot\frac{\partial f_3}{\partial x^1}\right)\left(\displaystyle \frac{1}{f_3}\cdot\frac{\partial f_3}{\partial x^2}\right)
     \end{aligned}
  \right.\ .
\end{equation}

All the functions that appear in the sequel are considered to be smooth and arbitrary unless otherwise specified.

We denote by $l_i:=\frac{\displaystyle 1}{\displaystyle f_3}\cdot\frac{\displaystyle \partial f_3}{\displaystyle \partial x^i}$ for $i\in \{1,2\}$, and we will first analyse the equations of the system.

1. For the first equation:\\
a) Any maximal connected set on which $\frac{\displaystyle \partial f_3}{\displaystyle \partial x^1}(x^1, x^2)\neq 0$ everywhere is of the form $I_1\times J_2$, where $J_2$ is an open subinterval in $I_2$. In such a case, we have
$\frac{\displaystyle 1}{\displaystyle f_3}\cdot\frac{\displaystyle \partial f_3}{\displaystyle \partial x^1}(x^1, x^2)=\frac{\displaystyle -1}{\displaystyle x^1-F(x^2)}$ and
$f_3(x^1, x^2)=\frac{\displaystyle G(x^2)}{\displaystyle x^1-F(x^2)}$ for any $(x^1, x^2)\in I_1\times J_2$, with $G(x^2)\neq 0$ and $F(x^2)\notin I_1$ for any $x^2\in J_2$.\\
b) Any maximal connected set on which $\frac{\displaystyle \partial f_3}{\displaystyle \partial x^1}=0$ is of the form $I_1\times K_2$, where $K_2$ is a subinterval in $I_2$ (possibly even trivial), closed with respect to $I_2$. In such a case, we have $f_3=f_3(x^2)$ on $I_1\times K_2$.

2. For the second equation:\\
a) Any maximal connected set on which $\frac{\displaystyle \partial f_3}{\displaystyle \partial x^2}(x^1, x^2)\neq 0$ everywhere is of the form $J_1\times I_2$, where $J_1$ is an open subinterval in $I_1$. In such a case, we have
$\frac{\displaystyle 1}{\displaystyle f_3} \cdot\frac{\displaystyle \partial f_3}{\displaystyle \partial x^2}(x^1, x^2)=\frac{\displaystyle -1}{\displaystyle x^2-H(x^1)}$ and
$f_3(x^1, x^2)=\frac{\displaystyle E(x^1)}{\displaystyle x^2-H(x^1)}$ for any $(x^1, x^2)\in J_1\times I_2$, with $E(x^1)\neq 0$ and $H(x^1)\notin I_2$ for any $x^1\in J_1$.\\
b) Any maximal connected set on which $\frac{\displaystyle \partial f_3}{\displaystyle \partial x^2}=0$ is of the form $K_1\times I_2$, where $K_1$ is a subinterval in $I_1$ (possibly even trivial), closed with respect to $I_1$. In such a case, we have $f_3=f_3(x^1)$ on $K_1\times I_2$.

We have the following cases.

\vspace{5pt}
Case I: There exists $(x_0^1, x_0^2)\in I_1\times I_2$ such that $\frac{\displaystyle \partial f_3}{\displaystyle \partial x^2}(x_0^1, x_0^2)\neq 0$. Then, $\frac{\displaystyle \partial f_3}{\displaystyle \partial x^2}(x^1, x^2)\neq 0$, hence $l_2(x^1, x^2)\neq 0$, everywhere on a maximal open interval $J_1\times I_2$. From the third equation of (\ref{T3.3}), we have $\frac{\displaystyle 1}{\displaystyle l_2} \cdot\frac{\displaystyle \partial l_2}{\displaystyle \partial x^1}=\frac{\displaystyle 1}{\displaystyle f_3}\cdot\frac{\displaystyle \partial f_3}{\displaystyle \partial x^1}$, hence $l_2(x^1, x^2)=P(x^2)f_3(x^1, x^2)$ for any $(x^1, x^2)\in J_1\times I_2$, with $P(x^2)\neq 0$ for any $x^2\in I_2$, but, from the second equation of (\ref{T3.3}), we have
$l_2(x^1, x^2)=\frac{\displaystyle -1}{\displaystyle x^2-H(x^1)}$, with $H(x^1)\notin I_2$ for any $x^1\in J_1$. We get $f_3(x^1, x^2)=\frac{\displaystyle -1}{\displaystyle P(x^2)(x^2-H(x^1))}$. From
the second equation of (\ref{T3.3}), we also have $f_3(x^1, x^2)=\frac{\displaystyle E(x^1)}{\displaystyle x^2-H(x^1)}$, so $E(x^1)P(x^2)=-1$ for any $(x^1, x^2)\in J_1\times I_2$. It follows that $E$ is constant on $J_1$; hence,
$f_3(x^1, x^2)=\frac{\displaystyle k_3}{\displaystyle x^2-H(x^1)}$, with $k_3\neq 0$, and $\frac{\displaystyle \partial f_3}{\displaystyle \partial x^2}(x^1, x^2)=\frac{\displaystyle -k_3}{\displaystyle (x^2-H(x^1))^2}$.

If there exists $x_1^1\in I_1$ a boundary point of $J_1$, then
$\lim_{x^1\rightarrow x_1^1,\ x^1\in J_1}\frac{\displaystyle \partial f_3}{\displaystyle \partial x^2}(x^1, x^2)=0$, so $\lim_{x^1\rightarrow x_1^1,\ x^1\in J_1}| x^2-H(x^1)|=\infty$, and $f_3(x_1^1, x^2)=\lim_{x^1\rightarrow x_1^1,\ x^1\in J_1}f_3(x^1, x^2)=0$, contradiction. We conclude that $J_1=I_1$, so
$\frac{\displaystyle \partial f_3}{\displaystyle \partial x^2}(x^1, x^2)\neq 0$ everywhere on $I_1\times I_2$, and
$$f_3(x^1, x^2)=\frac{\displaystyle k_3}{\displaystyle x^2-H(x^1)}\ \textrm{ for any }(x^1, x^2)\in I_1\times I_2,$$
with $k_3\neq 0$, $H(x^1)\notin I_2$ for any $x^1\in I_1$.

\vspace{5pt}
Subcase I.1: There exists $(x_1^1, x_1^2)\in I_1\times I_2$ such that $\frac{\displaystyle \partial f_3}{\displaystyle \partial x^1}(x_1^1, x_1^2)\neq \nolinebreak 0$. Then, $\frac{\displaystyle \partial f_3}{\displaystyle \partial x^1}(x^1, x^2)\neq 0$, hence $l_1(x^1, x^2)\neq 0$, everywhere on a maximal open interval $I_1\times J_2$. From the first equation of (\ref{T3.3}), we have
$l_1(x^1, x^2)=\frac{\displaystyle -1}{\displaystyle x^1-F(x^2)}$ and
$f_3(x^1, x^2)=\frac{\displaystyle G(x^2)}{\displaystyle x^1-F(x^2)}$  for any $(x^1, x^2)\in I_1\times J_2$, with $G(x^2)\neq 0$ and $F(x^2)\notin I_1$ for any $x^2\in J_2$. Hence, $k_3(x^1-F(x^2))=G(x^2)(x^2-H(x^1))$, and we get $-G(x^2)H'(x^1)=k_3\neq 0$ on $I_1\times J_2$, from which, $H'$ is constant on $I_1$, and $G$ is constant on $J_2$. Denoting $H'(x^1)=:c_1\neq 0$ on $I_1$, we obtain $G(x^2)=\frac{\displaystyle -k_3}{\displaystyle c_1}$ on $J_2$ and $H(x^1)=c_1x^1+c_2$ on $I_1$, with $c_2\in \mathbb R$. We get
$F(x^2)=\frac{\displaystyle x^2-c_2}{\displaystyle c_1}$ and
$f_3(x^1, x^2)=\frac{\displaystyle k_3}{\displaystyle x^2-c_1x^1-c_2}$ for any $(x^1, x^2)\in I_1\times J_2$, with $c_1x^1+c_2\notin I_2$ for any $x^1\in I_1$. Since $\frac{\displaystyle\partial f_3}{\displaystyle\partial x^1}(x^1, x^2)=\frac{\displaystyle c_1k_3}{\displaystyle (x^2-c_1x^1-c_2)^2}$, it follows that $J_2=I_2$; hence, $\frac{\displaystyle \partial f_3}{\displaystyle \partial x^1}(x^1, x^2)\neq 0$ everywhere on $I_1\times I_2$, and
$$f_3(x^1, x^2)=\frac{\displaystyle k_3}{\displaystyle x^2-c_1x^1-c_2}\ \textrm{ for any }(x^1, x^2)\in I_1\times I_2,$$
with $c_1, k_3\neq 0, c_2\in \mathbb R$ such that $x^2-c_1x^1-c_2\neq 0$ for any $(x^1, x^2)\in I_1\times I_2$, formula of $f_3$ that satisfies (\ref{T3.3}).

\vspace{5pt}
Subcase I.2: $\frac{\displaystyle \partial f_3}{\displaystyle \partial x^1}=0$ on $I_1\times I_2$. Since $f_3(x^1, x^2)=\frac{\displaystyle k_3}{\displaystyle x^2-H(x^1)}$, the condition is equivalent to $H$ is constant on $I_1$, and we get
$f_3(x^1, x^2)=\frac{\displaystyle k_3}{\displaystyle x^2-c_1}$ for any $(x^1, x^2)\in I_1\times I_2$, with $c_1\notin I_2$, formula that satisfies (\ref{T3.3}).

\vspace{5pt}
Case II: $\frac{\displaystyle \partial f_3}{\displaystyle \partial x^2}=0$ on $I_1\times I_2$, so $f_3=f_3(x^1)$ on $I_1\times I_2$.

\vspace{5pt}
Subcase II.1: $\frac{\displaystyle \partial f_3}{\displaystyle \partial x^1}(x_0^1, x_0^2)\neq 0$ for some $(x_0^1, x_0^2)\in I_1\times I_2$. Then,
$\frac{\displaystyle \partial f_3}{\displaystyle \partial x^1}(x^1, x^2)\neq 0$ everywhere on a maximal open interval $I_1\times J_2$. We have
$\frac{\displaystyle 1}{\displaystyle f_3}\cdot\frac{\displaystyle \partial f_3}{\displaystyle \partial x^1}(x^1, x^2)=\frac{\displaystyle -1}{\displaystyle x^1-F(x^2)}$ on $I_1\times J_2$, with $F(x^2)\notin I_1$ for any $x^2\in J_2$. But $f_3=f_3(x^1)$, which implies that $F$ is constant on $J_2$, $F(x^2)=c_1\notin I_1$, so $f_3(x^1, x^2)=\frac{\displaystyle G(x^2)}{\displaystyle x^1-c_1}$ on $I_1\times J_2$. Since $f_3=f_3(x^1)$, we get $G$ constant on $J_2$, so $f_3(x^1, x^2)=\frac{\displaystyle k_3}{\displaystyle x^1-c_1}$ on $I_1\times J_2$, with $k_3\neq 0$. From $\frac{\displaystyle \partial f_3}{\displaystyle \partial x^1}(x^1, x^2)=\nolinebreak\frac{\displaystyle -k_3}{\displaystyle (x^1-c_1)^2}$, we get $J_2=I_2$. We obtain $f_3(x^1, x^2)=\frac{\displaystyle k_3}{\displaystyle x^1-c_1}$ for any $(x^1, x^2)\in I_1\times I_2$, which satisfies (\ref{T3.3}).

\vspace{5pt}
Subcase II.2: $\frac{\displaystyle \partial f_3}{\displaystyle \partial x^1}=0$ on $I_1\times I_2$. Then $f_3=f_3(x^2)$, so $f_3$ is constant on $I_1\times I_2$, which satisfies (\ref{T3.3}).

\vspace{5pt}
We notice that all the obtained solutions $f_3$ can be expressed by the formula
$$f_3(x^1,x^2)=\frac{\displaystyle 1}{\displaystyle c_1x^1+c_2x^2+c_3},$$
with $c_1, c_2, c_3\in \mathbb R$ such that $c_1x^1+c_2x^2+c_3\neq 0$ for any $(x^1,x^2)\in I_1\times I_2$.
\end{proof}

\begin{example}
The warped product manifold
$$\Big((0, \infty)\times (0, \infty)\Big)\times_{f}\mathbb R=\Big((0, \infty)\times (0, \infty)\times \mathbb R,\ g=(dx^1)^2+(dx^2)^2+f^2(dx^3)^2\Big),$$ for
$$f(x^1,x^2)=x^1+x^2+m,$$
with $m\geq 0$, is a flat Riemannian manifold.
\end{example}

And we can further deduce
\begin{corollary}\label{corT3.8x}
There do not exist proper flat warped product manifolds of the form
$$\mathbb R^2\times_{f}\mathbb R=\Big(\mathbb R^3,\ g=(dx^1)^2+(dx^2)^2+f^2(dx^3)^2\Big).$$
\end{corollary}

\begin{theorem}\label{ps11mac6}
If $f_1=1$, $f_2=f_2(x^1)$, $f_3=f_3(x^1)$, then the biwarped product manifold
$$I_1\times_{\frac{1}{f_2}} I_2\times_{\frac{1}{f_3}} I_3=(I,g)$$
is a flat Riemannian
manifold if and only if
one of the following assertions holds:

(1) $f_2=k_2\in \mathbb R\setminus \{0\}$ and $f_3=k_3\in \mathbb R\setminus \{0\}$;

(2) $f_2=k_2\in \mathbb R\setminus \{0\}$ and $f_3(x^1)=\frac{\displaystyle k_3}{\displaystyle x^1-c_3}$, with $k_3\in \mathbb R\setminus \{0\}$, $c_3\in \mathbb R\setminus I_1$;

(3) $f_2(x^1)=\frac{\displaystyle k_2}{\displaystyle x^1-c_2}$ and $f_3=k_3\in \mathbb R\setminus \{0\}$, with $k_2\in \mathbb R\setminus \{0\}$, $c_2\in \mathbb R\setminus I_1$.

In the first case from above, the manifold is just a direct product, and in
the last two cases, the manifold reduces to a warped product manifold.
\end{theorem}
\begin{proof}
We have: $a_{21}=\frac{\displaystyle f_2'}{\displaystyle f_2}$, $a_{31}=\frac{\displaystyle f_3'}{\displaystyle f_3}$, $a_{12}=a_{13}=a_{23}=a_{32}=0$, and we get:
\begin{align*}
\R(E_1,E_1)&=E_1(a_{21})+E_1(a_{31})-a_{21}^2-a_{31}^2,\\
\R(E_2,E_2)&=E_1(a_{21})-a_{21}^2-a_{21}a_{31},\\
\R(E_3,E_3)&=E_1(a_{31})-a_{31}^2-a_{21}a_{31},\\
\R(E_1,E_2)&=\R(E_1,E_3)=\R(E_2,E_3)=0.
\end{align*}

Then, $\R=0$ if and only if
\begin{equation*}
\left\{
    \begin{aligned}
&\left(\frac{\displaystyle f_2'}{\displaystyle  f_2}\right)'=\left(\frac{\displaystyle f_2'}{\displaystyle  f_2}\right)^2\\
&\left(\frac{\displaystyle f_3'}{\displaystyle  f_3}\right)'=\left(\frac{\displaystyle f_3'}{\displaystyle  f_3}\right)^2\\
& f_2'f_3'=0
     \end{aligned}
  \right.,
\end{equation*}
which gives
\begin{equation*}
\left\{
    \begin{aligned}
&h_2'=h_2^2\\
&h_3'=h_3^2\\
&h_2h_3=0
    \end{aligned}
  \right.\ .
\end{equation*}

The first equation is equivalent to $h_2(x^1)=\displaystyle\frac{-1}{x^1-c_2}$, with $c_2\notin I_1$, or $h_2=0$ on $I_1$.

The second equation is equivalent to $h_3(x^1)=\displaystyle\frac{-1}{x^1-c_3}$, with $c_3\notin I_1$, or $h_3=0$ on $I_1$.

Due to the third equation, we obtain $h_2=0$ or $h_3=0$. So, we have the cases:

I. $h_2=0$ and $h_3=0$, that is, $f_2=k_2$ and $f_3=k_3$, with $k_2, k_3\neq 0$;

II. $h_2=0$ and $h_3(x^1)=\displaystyle\frac{-1}{x^1-c_3}$, with $c_3\notin I_1$; that is, $f_2=k_2$ and $f_3(x^1)=\displaystyle\frac{k_3}{x^1-c_3}$, with $k_2, k_3\neq 0, c_3\notin I_1$;

III. $h_2(x^1)=\displaystyle\frac{-1}{x^1-c_2}$, with $c_2\notin I_1$, and $h_3=0$; that is, $f_2(x^1)=\displaystyle\frac{k_2}{x^1-c_2}$ and $f_3=k_3$, with $k_2, k_3\neq 0, c_2\notin I_1$.
\end{proof}

We can further deduce
\begin{corollary}\label{corT3.8w}
There do not exist proper flat biwarped product manifolds of the form
$$I_1\times_{f} I_2\times_{h} I_3=\Big(I,\ g=(dx^1)^2+f^2(dx^2)^2+h^2(dx^3)^2\Big).$$
\end{corollary}

\begin{theorem}\label{ps11mv}
If $f_1=1$, $f_2=f_2(x^1)$, $f_3=f_3(x^1,x^2)$, then the sequential warped product manifold
$$\left(I_1\times_{\frac{1}{f_2}}I_2\right)\times_{\frac{1}{f_3}}I_3=(I,g)$$
is a flat Riemannian manifold if and only if one of the following assertions holds:

(1) $f_2=k_2\in \mathbb R\setminus \{0\}$ and $f_3=k_3\in \mathbb R\setminus \{0\}$;

(2) $f_2=k_2\in \mathbb R\setminus \{0\}$ and $f_3(x^1)=\frac{\displaystyle k_3}{\displaystyle x^1-c_3}$, with $k_3 \in \mathbb R\setminus \{0\}$, $c_3\in \mathbb{R} \setminus I_1$;

(3) $f_2=k_2\in \mathbb R\setminus \{0\}$ and $f_3(x^2)=\frac{\displaystyle k_3}{\displaystyle x^2-c_3}$, with $k_3 \in \mathbb R\setminus \{0\}$, $c_3\in \mathbb{R} \setminus I_2$;

(4) $f_2=k_2\in \mathbb R\setminus \{0\}$ and $f_3(x^1,x^2)=\frac{\displaystyle k_3}{\displaystyle c_1 x^1+x^2+c_2}$, with $k_3$, $c_1\in\nolinebreak \mathbb R\setminus\nolinebreak \{0\}$, $c_2\in \mathbb R$ such that $c_1x^1+x^2+c_2\neq 0$ for any $(x^1,x^2)\in I_1\times I_2$;

(5) $f_2(x^1)=\frac{\displaystyle k_2}{\displaystyle x^1-c_2}$ and $f_3=k_3\in \mathbb R\setminus \{0\}$, with $k_2\in \mathbb R\setminus \{0\}$, $c_2\in \mathbb R\setminus I_1$;

(6) $f_2(x^1)=\frac{\displaystyle k_2}{\displaystyle x^1-c_2}$ and
$f_3(x^1,x^2)=\frac{\displaystyle k_3}{\displaystyle (x^1-c_2)\cos \left(\frac{x^2}{k_2}-c_3\right)}$, with $k_2, k_3 \in \mathbb R\setminus \{0\}$, $c_2\in \mathbb R\setminus I_1$, $c_3 \in \mathbb R$ such that $\left|\displaystyle\frac{x^2}{k_2}-c_3\right|<\displaystyle\frac{\pi}{2}$ for any $x^2 \in I_2$;

(7) $f_2(x^1)=\frac{\displaystyle k_2}{\displaystyle x^1-c_2}$ and
$f_3(x^1,x^2)=\frac{\displaystyle k_3}{\displaystyle 1+c_1(x^1-c_2)\cos \left(\frac{x^2}{k_2}-c_3\right)}$, with $c_1$, $k_2$, $k_3 \in \mathbb R\setminus \{0\}$, $c_2\in \mathbb R\setminus I_1$, $c_3 \in \mathbb R$ such that
$$\displaystyle 1+c_1(x^1-c_2)\cos \left(\frac{x^2}{k_2}-c_3\right)\neq 0\ \textrm{ for any }(x^1, x^2) \in I_1\times I_2.$$

In the first case from above, the manifold is just a direct product, and in the next four cases, the manifold reduces to a warped product manifold.
\end{theorem}

\begin{proof}
We have: $a_{21}=\frac{\displaystyle f_2'}{\displaystyle f_2}$, $a_{31}=\frac{\displaystyle 1}{\displaystyle f_3}\cdot\frac{\displaystyle \partial f_3}{\displaystyle \partial x^1}$, $a_{32}=\frac{\displaystyle f_2}{\displaystyle f_3}\cdot\frac{\displaystyle \partial f_3}{\displaystyle \partial x^2}$, and $a_{12}=a_{13}=a_{23}=0$, and we get:
\begin{align*}
\R(E_1,E_1)&=E_1(a_{21})+E_1(a_{31})-a_{21}^2-a_{31}^2,\\
\R(E_2,E_2)&=E_1(a_{21})+E_2(a_{32})-a_{21}^2-a_{32}^2-a_{21}a_{31},\\
\R(E_3,E_3)&=E_1(a_{31})+E_2(a_{32})-a_{31}^2-a_{32}^2-a_{21}a_{31},\\
\R(E_1,E_2)&=E_1(a_{32})-a_{31}a_{32},\\
\R(E_1,E_3)&=\R(E_2,E_3)=0.
\end{align*}

Then, $\R=0$ if and only if
\begin{equation*}
\left\{
    \begin{aligned}
&      E_1(a_{21})-a_{21}^2=-\left(E_1(a_{31})-a_{31}^2\right)\\
&E_1(a_{21})-a_{21}^2=-\left(E_2(a_{32})-a_{32}^2\right)+a_{21}a_{31}\\
&E_1(a_{31})-a_{31}^2=-\left(E_2(a_{32})-a_{32}^2\right)+a_{21}a_{31}\\
&E_1(a_{32})=a_{31}a_{32}
    \end{aligned}
  \right. ,
\end{equation*}
which is equivalent to
\begin{equation}
\label{T3.8}
\left\{
\begin{aligned}
&a_{21}'=a_{21}^2\\
\vspace{0.1cm}
&\frac{\displaystyle \partial a_{31}}{\displaystyle \partial x^1}=a_{31}^2\\
\vspace{0.1cm}
&f_2\frac{\displaystyle \partial a_{32}}{\displaystyle \partial x^2}-a_{32}^2=a_{21}a_{31}\\
&\frac{\displaystyle \partial a_{32}}{\displaystyle \partial x^1}=a_{31}a_{32}
\end{aligned}
\right. .
\end{equation}

All the functions that appear in the sequel are considered to be smooth and arbitrary unless otherwise specified.

We will start by analysing the equations of the system.

1. The first equation has the solutions:\\
a) $a_{21}=0$ on $I_1$, which corresponds to $f_2$ constant on $I_1$;\\
b) $a_{21}(x^1)=\frac{\displaystyle -1}{\displaystyle x^1-c_2}$ for any $x^1\in I_1$, which corresponds to $f_2(x^1)=\frac{\displaystyle k_2}{\displaystyle x^1-c_2}$ for any $x^1\in I_1$, with $k_2\neq 0$, $c_2\notin I_1$.

2. For the second equation:\\
a) Any maximal connected set on which $a_{31}(x^1, x^2)\neq 0$ everywhere is of the form $I_1\times J_2$, where $J_2$ is an open subinterval in $I_2$. In such a case, we have $a_{31}(x^1, x^2)=\frac{\displaystyle -1}{\displaystyle x^1-F(x^2)}$ and
$f_3(x^1, x^2)=\frac{\displaystyle G(x^2)}{\displaystyle x^1-F(x^2)}$ for any $(x^1, x^2)\in\nolinebreak I_1\times\nolinebreak J_2$, with $G(x^2)\neq 0$ and $F(x^2)\notin I_1$ for any $x^2\in J_2$.\\
b) Any maximal connected set on which $a_{31}=0$ is of the form $I_1\times K_2$, where $K_2$ is a subinterval in $I_2$ (possibly even trivial), closed with respect to $I_2$. In such a case, we have $\frac{\displaystyle \partial f_3}{\displaystyle \partial x^1}(x^1, x^2)=0$ for any $(x^1, x^2)\in I_1\times K_2$.

3. For the fourth equation:\\
a) Any maximal connected set on which $a_{32}(x^1, x^2)\neq 0$ everywhere is of the form $I_1\times J_2^2$, where $J_2^2$ is an open subinterval in $I_2$. In such a case, we have $a_{32}(x^1, x^2)=G_2(x^2)f_3(x^1, x^2)$ and $f_3(x^1, x^2)=\frac{\displaystyle f_2(x^1)}{\displaystyle M(x^1)-G_3(x^2)}$ for any $(x^1, x^2)\in I_1\times J_2^2$, where $G_3$ is an antiderivative of $G_2$, $G_2(x^2)\neq 0$, $M(x^1)-G_3(x^2)\neq 0$ for any $x^1\in I_1$, $x^2\in J_2^2$.\\
b) Any maximal connected set on which $a_{32}=0$ is of the form $I_1\times K_2^2$, where $K_2^2$ is a subinterval in $I_2$ (possibly even trivial), closed with respect to $I_2$. In such a case, we have $\frac{\displaystyle \partial f_3}{\displaystyle \partial x^2}(x^1, x^2)=0$ for any $(x^1, x^2)\in I_1\times K_2^2$, which implies $f_3=f_3(x^1)$ on $I_1\times K_2^2$.

In fact, because $\frac{\displaystyle a_{32}(x^1, x^2)}{\displaystyle f_3(x^1, x^2)}=\left\{
\begin{aligned}
&G_2(x^2)\neq 0 \hspace*{-8 pt}&,\ & x^2\in J_2^2\\
\vspace{0.1cm}
&0 &,\ & x^2\in K_2^2
\end{aligned}
\right.\ $, we get $\frac{\displaystyle \partial}{\displaystyle \partial x^1}\left (\frac{\displaystyle a_{32}}{\displaystyle f_3}\right )=0$ on $I_1\times\nolinebreak I_2$, i.e., $\frac{\displaystyle a_{32}}{\displaystyle f_3}=\frac{\displaystyle a_{32}}{\displaystyle f_3}(x^2)$. We define the function $\bar{G}_2:I_2\rightarrow \mathbb R$, $\bar{G}_2(x^2)=\frac{\displaystyle a_{32}(x^1, x^2)}{\displaystyle f_3(x^1, x^2)}$, $x^2\in I_2$, for an arbitrary $x^1\in I_1$. The function is well defined, smooth, and $\bar{G}_2(x^2)$ is zero on the sets $K_2^2$ and nonzero on the intervals $J_2^2$. Let $G_3$ be an antiderivative of $\bar{G}_2$ on $I_2$. We notice that $G_3$ is constant on any set $K_2^2$ and
$\frac{\displaystyle \partial}{\displaystyle \partial x^2}\left (G_3(x^2)+\frac{\displaystyle f_2(x^1)}{\displaystyle f_3(x^1, x^2)}\right )=0$ on $I_1\times I_2$.

We consider now $M:I_1\rightarrow \mathbb R$, $M(x^1)=G_3(x^2)+\frac{\displaystyle f_2(x^1)}{\displaystyle f_3(x^1, x^2)}$, $x^1\in I_1$, for an arbitrary $x^2\in I_2$. The function $M$ is well defined, smooth, and we have  $M(x^1)-G_3(x^2)\neq 0$ and
\begin{equation}\label{f_3}
f_3(x^1, x^2)=\frac{\displaystyle f_2(x^1)}{\displaystyle M(x^1)-G_3(x^2)}
\end{equation}
for any $(x^1, x^2)\in I_1\times I_2$, and $G_3'(x^2)$ is zero on the sets $K_2^2$ and nonzero on the intervals $J_2^2$. We get $a_{32}(x^1, x^2)=\bar{G}_2 (x^2)f_3(x^1, x^2)$ on $I_1\times I_2$.

\bigskip
Depending on the zero or nonzero values of $a_{21}(x^1)$, $a_{31}(x^1, x^2)$, and $a_{32}(x^1, x^2)$, we have the following cases.

\vspace{5pt}
Case I: $a_{21}=0$ on $I_1$. In this case, $f_2=k_2\in \mathbb R\setminus \{0\}$ and the third equation of (\ref{T3.8}) becomes
$\frac{\displaystyle \partial}{\displaystyle \partial x^2}\left (\frac{\displaystyle 1}{\displaystyle f_3}\cdot \frac{\displaystyle \partial f_3}{\displaystyle \partial x^2}\right )=\left (\frac{\displaystyle 1}{\displaystyle f_3}\cdot\frac{\displaystyle \partial f_3}{\displaystyle \partial x^2}\right )^2$.
We get:

(i) $\frac{\displaystyle 1}{\displaystyle f_3}\cdot \frac{\displaystyle \partial f_3}{\displaystyle \partial x^2}=\frac{\displaystyle -1}{\displaystyle x^2-H(x^1)}$, which corresponds to
$f_3(x^1, x^2)=\frac{\displaystyle M_2(x^1)}{\displaystyle x^2-H(x^1)}$, on any $I_1\times J_2^2$ -type set, with $M_2(x^1)\neq 0$ and $H(x^1)\notin J_2^2$ for any $x^1\in I_1$. Combining with (\ref{f_3}) and differentiating with respect to $x^2$, we get $M_2(x^1)G_3'(x^2)=-k_2$; hence, $G_3'$ is a nonzero constant on any interval $J_2^2$, and $M_2$ is constant on $I_1$. Denoting $G_3'=:c_3\neq 0$ on  $J_2^2$, we get $M_2=\frac{\displaystyle -k_2}{\displaystyle c_3}$ on $I_1$, and $G_3(x^2)=c_3 x^2+d_3$ on $J_2^2$, $d_3\in \mathbb R$, from which, $M(x^1)=c_3H(x^1)+d_3$ on $I_1$. We obtain
$f_3(x^1, x^2)=\frac{\displaystyle k_2}{\displaystyle c_3(H(x^1)-x^2)}$ on $I_1\times J_2^2$.

(ii) $\frac{\displaystyle \partial f_3}{\displaystyle \partial x^2}=0$ on any $I_1\times K_2^2$ -type set, and $G_3'(x^2)=0$ on $K_2^2$.

Resuming, due to the continuity of $G_3'$, we have $J_2^2=I_2$ or $K_2^2=I_2$. Hence, $G_3'(x^2)=\nolinebreak c_3$ on $I_2$, with $c_3\in \mathbb R$. For $c_3\neq 0$, we have $f_3(x^1, x^2)=\frac{\displaystyle k_2}{\displaystyle c_3(H(x^1)-x^2)}$, $M(x^1)=c_3H(x^1)+d_3$, $d_3\in \mathbb R$, and $a_{32}(x^1, x^2)\neq 0$ everywhere on $I_1\times I_2$. For $c_3=0$, we have $f_3=f_3(x^1)$ and $a_{32}=0$ on $I_1\times I_2$.

\vspace{5pt}
Subcase I.1: There exists $(x_0^1, x_0^2)\in I_1\times I_2$ such that $a_{31}(x_0^1, x_0^2)\neq 0$. Then, $a_{31}(x^1, x^2)\neq 0$ everywhere on a maximal open interval $I_1\times J_2$. We get
$a_{31}(x^1, x^2)=\frac{\displaystyle -1}{\displaystyle x^1-F(x^2)}$ and
$f_3(x^1, x^2)=\frac{\displaystyle G(x^2)}{\displaystyle x^1-F(x^2)}$ on $I_1\times J_2$, with $F(x^2)\notin I_1$ and $G(x^2)\neq 0$ for any $x^2\in\nolinebreak J_2$. From (\ref{f_3}), we have
$f_3(x^1, x^2)=\frac{\displaystyle k_2}{\displaystyle M(x^1)-G_3(x^2)}$, so
$k_2(x^1-\nolinebreak F(x^2))=G(x^2)(M(x^1)-G_3(x^2))$, which implies $k_2=G(x^2)M'(x^1)$ for any $(x^1, x^2)\in I_1\times J_2$. Hence, $G$ is constant on $J_2$, and $M'$ is a nonzero constant on $I_1$. Denoting $G=k_3\in \mathbb R\setminus \{0\}$, we get $f_3(x^1, x^2)=\frac{\displaystyle k_3}{\displaystyle x^1-F(x^2)}$ for any $(x^1, x^2)\in I_1\times J_2$. We notice that with this formula, the second equation of (\ref{T3.8}) is verified. Supposing $J_2\neq I_2$, from the expression of $a_{31}(x^1, x^2)$, we deduce that $F$ is unbounded on $J_2$; hence, $f_3$ vanishes at a boundary point of $I_1\times J_2$, contradiction. So, $J_2=I_2$, and $a_{31}(x^1, x^2)\neq 0$ everywhere on $I_1\times I_2$.

Also, it follows that, if $a_{31}$ vanishes at a point of $I_1\times I_2$, then $a_{31}=0$ on $I_1\times I_2$.

\vspace{5pt}
Subsubcase I.1.1: $a_{32}(x^1, x^2)\neq 0$ everywhere on $I_1\times I_2$. We have $f_3(x^1, x^2)=\frac{\displaystyle k_2}{\displaystyle c_3(H(x^1)-x^2)}=\frac{\displaystyle k_3}{\displaystyle x^1-F(x^2)}$ on $I_1\times I_2$, from which, by differentiation with respect to $x^2$, $k_2 F'(x^2)=k_3 c_3$, so $F(x^2)=\frac {\displaystyle k_3 c_3}{\displaystyle k_2}x^2-c_4$ on $I_2$, with $c_4\in \mathbb R$. We get $f_3(x^1, x^2)=\frac {\displaystyle k_3}{\displaystyle x^1+c_5 x^2+ c_4}$ on $I_1\times I_2$, with $c_5\neq 0$ and $c_4\in \mathbb R$, formula that satisfies (\ref{T3.8}).

\vspace{5pt}
Subsubcase I.1.2: $a_{32}=0$ on $I_1\times I_2$. We have $f_3(x^1, x^2)=\frac{\displaystyle k_3}{\displaystyle x^1-F(x^2)}$ and $f_3=f_3(x^1)$ on $I_1\times I_2$, so $F$ is constant on $I_2$, $F=:c_6\in \mathbb R\setminus I_1$, and $f_3(x^1, x^2)=\frac {\displaystyle k_3}{\displaystyle x^1-c_6}$ on $I_1\times I_2$, formula that satisfies (\ref{T3.8}).

\vspace{5pt}
Subcase I.2: $a_{31}=0$ on $I_1\times I_2$. Hence, $f_3=f_3(x^2)$ on $I_1\times I_2$.

\vspace{5pt}
Subsubcase I.2.1: $a_{32}(x^1, x^2)\neq 0$ everywhere on $I_1\times I_2$. We have $f_3(x^1, x^2)=\frac{\displaystyle k_2}{\displaystyle c_3(H(x^1)-x^2)}$ and $f_3=f_3(x^2)$ on $I_1\times I_2$; hence, $H$ is constant on $I_1$, $H=:c_7$. Renoting the constants, we get $f_3(x^1, x^2)=\frac {\displaystyle k_3}{\displaystyle x^2-c_7}$ on $I_1\times I_2$, with $k_3\neq 0$, $c_7\in \mathbb R \setminus I_2$, formula that satisfies (\ref{T3.8}).

\vspace{5pt}
Subsubcase I.2.2: $a_{32}=0$ on $I_1\times I_2$. Hence, $f_3=f_3(x^1)$ and $f_3=f_3(x^2)$, so $f_3$ is constant on $I_1\times I_2$, which satisfies (\ref{T3.8}).

\vspace{5pt}
Case II: $a_{21}(x^1)\neq 0$ everywhere on $I_1$. In this case, $a_{21}(x^1)=\frac{\displaystyle -1}{\displaystyle x^1-c_2}$, and $f_2(x^1)=\frac{\displaystyle k_2}{\displaystyle x^1-c_2}$ for any $x^1\in I_1$, with $k_2\neq 0$, $c_2\notin I_1$.

\vspace{5pt}
Subcase II.1: $a_{31}(x^1, x^2)\neq 0$ everywhere on a maximal open interval $I_1\times J_2$. We have
$a_{31}(x^1, x^2)=\frac{\displaystyle -1}{\displaystyle x^1-F(x^2)}$ and
$f_3(x^1, x^2)=\frac{\displaystyle G(x^2)}{\displaystyle x^1-F(x^2)}$ on $I_1\times J_2$, with $F(x^2)\notin I_1$ and $G(x^2)\neq 0$ for any $x^2\in\nolinebreak J_2$. The fourth equation of (\ref{T3.8}) becomes
$\frac{\displaystyle \partial a_{32}}{\displaystyle \partial x^1}=\frac{\displaystyle -1}{\displaystyle x^1-F(x^2)} a_{32}$, hence
$a_{32}(x^1, x^2)=\frac{\displaystyle H(x^2)}{\displaystyle x^1-F(x^2)}$. Since $\frac{\displaystyle \partial f_3}{\displaystyle \partial x^2}(x^1, x^2)=
\frac{\displaystyle G'(x^2)(x^1-F(x^2))+G(x^2)F'(x^2)}{\displaystyle (x^1-F(x^2))^2}$, the above equality becomes
$$G'(x^2)+\frac{\displaystyle G'(x^2)(c_2-F(x^2))+G(x^2)F'(x^2)}{\displaystyle x^1-c_2}=\frac{1}{\displaystyle k_2} G(x^2)H(x^2),$$
from which, $G'(x^2)(c_2-F(x^2))+G(x^2)F'(x^2)=0$ for any $x^2\in J_2$.

\vspace{5pt}
Subsubcase II.1.1: There exists $x_0^2\in J_2$ such that $F(x_0^2)\neq c_2$. Then,
$\frac{\displaystyle G'(x^2)}{\displaystyle G(x^2)}=\frac{\displaystyle (c_2-F(x^2))'}{\displaystyle c_2-F(x^2)}$, so $G(x^2)=c_8 (c_2-F(x^2))$, and $F(x^2)\neq c_2$ for any $x^2\in J_2$, with $c_8\neq 0$. We get $f_3(x^1, x^2)=c_8 \frac{\displaystyle c_2-F(x^2)}{\displaystyle x^1-F(x^2)}$ for any $(x^1, x^2)\in\nolinebreak I_1\times\nolinebreak J_2$, and the third equation of (\ref{T3.8}) becomes
$$\frac{\displaystyle F''(x^2)}{\displaystyle c_2-F(x^2)}+2 \left(\frac{\displaystyle F'(x^2)}{\displaystyle c_2-F(x^2)}\right)^2=\frac{\displaystyle -1}{\displaystyle k_2^2}.$$
Denoting $L=\frac{\displaystyle (c_2-F(x^2))'}{\displaystyle c_2-F(x^2)}$, the above equation becomes $L'-L^2=\frac{\displaystyle 1}{\displaystyle k_2^2}$. We get
$L(x^2)=\nolinebreak \frac{\displaystyle 1}{\displaystyle k_2} \tan\left(\frac{\displaystyle x^2+c_9}{\displaystyle k_2}\right)$ on $J_2$, with $c_9\in \mathbb R$ such that
$\frac{\displaystyle x^2+c_9}{\displaystyle k_2}\in\nolinebreak \left(-\frac{\displaystyle \pi}{\displaystyle 2}, \frac{\displaystyle \pi}{\displaystyle 2}\right)$ for any $x^2\in J_2$. We obtain
$F(x^2)=c_2+\frac{\displaystyle c_{10}}{\displaystyle \cos \left(\frac{\displaystyle x^2+c_9}{\displaystyle k_2}\right)}$, with $c_{10}\neq 0$, so
$$f_3(x^1, x^2)=\frac{\displaystyle G(x^2)}{\displaystyle x^1-F(x^2)}=
\frac{\displaystyle -c_8 c_{10}}{\displaystyle (x^1-c_2)\cos \left(\frac{\displaystyle x^2+c_9}{\displaystyle k_2}\right)-c_{10}}$$
for any $(x^1, x^2)\in I_1\times J_2.$
It follows that
$$a_{31}(x^1, x^2)=\frac{\displaystyle -\cos \left(\frac{\displaystyle x^2+c_9}{\displaystyle k_2}\right)}{\displaystyle (x^1-c_2)\cos \left(\frac{\displaystyle x^2+c_9}{\displaystyle k_2}\right)-c_{10}}$$
for any $(x^1, x^2)\in I_1\times J_2$.\\
If $x_0^2\in I_2$ is a boundary point of $J_2$, then
$$\lim_{x^2\rightarrow x_0^2,\ x^2\in J_2}a_{31}(x^1, x^2)=a_{31}(x^1, x_0^2)=0$$
for any $x^1\in I_1$, so
$\cos \left(\frac{\displaystyle x_0^2+c_9}{\displaystyle k_2}\right)=0$, and
$\frac{\displaystyle x_0^2+c_9}{\displaystyle k_2}\in \left\{-\frac{\displaystyle \pi}{\displaystyle 2}, \frac{\displaystyle \pi}{\displaystyle 2}\right\}$. It follows that
$\frac{\displaystyle \partial f_3}{\displaystyle \partial x^2}(x^1, x_0^2)= \mp\frac{\displaystyle c_8}{\displaystyle k_2 c_{10}}(x^1-c_2)\neq 0$ for any  $x^1\in I_1$, and $\frac{\displaystyle \partial^2 f_3}{\displaystyle \partial x^1\partial x^2}(x^1, x_0^2)= \mp\frac{\displaystyle c_8}{\displaystyle k_2 c_{10}}\neq 0$.

\vspace{5pt}
The interval $K_2$ which contains $x_0^2$ is trivial, that is, $K_2=\{x_0^2\}$ (otherwise, from $\frac{\displaystyle \partial f_3}{\displaystyle \partial x^1}(x^1, x^2)=0$ on $I_1\times K_2$, it follows that $\frac{\displaystyle \partial^2 f_3}{\displaystyle \partial x^2\partial x^1}(x^1, x^2)=0$ on $I_1\times K_2$, contradiction).

\vspace{5pt}
We also notice that the intervals of $J_2$ -type extend as much as the limits of $I_2$ permit till a length of $k_2\pi$. So, the $K_2$ -type intervals are all trivial, and they are borders, on both sides, of $J_2$ -type intervals.

Comparing the expressions of $f_3(x^1, x^2)$ at the left and at the right side of a trivial interval $K_2$, we notice that these can be written with the same values of the constants $c_2$, $c_8$, $c_9$, $c_{10}$, i.e., we have the same formula for $f_3(x^1, x^2)$ at both sides of $K_2$ and also in $K_2$, without the restriction to $\left(-\frac{\displaystyle \pi}{\displaystyle 2}, \frac{\displaystyle \pi}{\displaystyle 2}\right)$ of the argument of the cosine function, but with the condition that the denominator should be nonzero. We conclude that
$$f_3(x^1, x^2)=
\frac{\displaystyle c_8}{\displaystyle \frac{1}{-c_{10}}(x^1-c_2)\cos \left(\frac{\displaystyle x^2+c_9}{\displaystyle k_2}\right)+1} \ \textrm{ for any } (x^1, x^2)\in I_1\times I_2,$$
with $c_8, c_{10}\neq 0$, $c_2\in \mathbb R\setminus I_1$, $c_9\in \mathbb R$ such that
$\displaystyle \frac{1}{-c_{10}}(x^1-c_2)\cos \left(\frac{\displaystyle x^2+c_9}{\displaystyle k_2}\right)+1\neq 0$ everywhere on $I_1\times I_2$. This formula of $f_3$ satisfies (\ref{T3.8}).

\vspace{5pt}
Subsubcase II.1.2: $F(x^2)=c_2$ for any $x^2\in J_2$. Then,
$f_3(x^1, x^2)=\frac{\displaystyle G(x^2)}{\displaystyle x^1-c_2}$ on $I_1\times J_2$, with $G(x^2)\neq 0$ for any $x^2\in\nolinebreak J_2$. The third equation of (\ref{T3.8}) becomes
$$k_2^2\left[\left(\frac{\displaystyle G'(x^2)}{\displaystyle G(x^2)}\right)'-\left(\frac{\displaystyle G'(x^2)}{\displaystyle G(x^2)}\right)^2\right]=1,$$
from which,
$$\frac{\displaystyle G'(x^2)}{\displaystyle G(x^2)}=\frac{\displaystyle 1}{\displaystyle k_2}\tan\left(\frac{\displaystyle x^2+c_{11}}{\displaystyle k_2}\right),$$
with $c_{11}\in \mathbb R$ such that
$\frac{\displaystyle x^2+c_{11}}{\displaystyle k_2}\in\nolinebreak \left(-\frac{\displaystyle \pi}{\displaystyle 2}, \frac{\displaystyle \pi}{\displaystyle 2}\right)$ for any $x^2\in J_2$. We obtain
$G(x^2)=\frac{\displaystyle c_{12}}{\displaystyle \cos \left(\frac{\displaystyle x^2+c_{11}}{\displaystyle k_2}\right)}$, with $c_{12}\neq 0$, so
$$f_3(x^1, x^2)=\frac{\displaystyle G(x^2)}{\displaystyle x^1-F(x^2)}=
\frac{\displaystyle c_{12}}{\displaystyle (x^1-c_2)\cos \left(\frac{\displaystyle x^2+c_{11}}{\displaystyle k_2}\right)},$$
and $a_{31}(x^1, x^2)=\frac{\displaystyle -1}{\displaystyle x^1-c_2}$ for any $(x^1, x^2)\in I_1\times J_2$. Since $J_2$ is maximal with $a_{31}\neq 0$, we obtain $J_2=I_2$. Hence, subsubcase II.1.2 is valid on $I_1\times I_2$. This formula of $f_3$ satisfies (\ref{T3.8}).

\vspace{5pt}
Subcase II.2: $a_{31}=0$ on $I_1\times I_2$, that is, $\frac{\displaystyle\partial f_3}{\displaystyle\partial x^1}=0$ on $I_1\times I_2$. We get
$$\frac{\displaystyle\partial^2 f_3}{\displaystyle\partial x^1\partial x^2}= \frac{\displaystyle\partial^2 f_3}{\displaystyle\partial x^2\partial x^1}=0,$$
and
$$\frac{\displaystyle\partial a_{32}}{\displaystyle\partial x^1}(x^1, x^2)= \frac{\displaystyle k_2}{\displaystyle x^1-c_2} \cdot\frac{\frac{\displaystyle\partial^2 f_3}{\displaystyle\partial x^1\partial x^2}}{\displaystyle f_3}- \frac{\displaystyle k_2}{\displaystyle (x^1-c_2)^2}\cdot\frac{\frac{\displaystyle\partial f_3}{\displaystyle\partial x^2}}{\displaystyle f_3}=
-\frac{\displaystyle k_2}{\displaystyle (x^1-c_2)^2} \cdot\frac{\frac{\displaystyle\partial f_3}{\displaystyle\partial x^2}}{\displaystyle f_3}$$
on $I_1\times I_2$. The fourth equation in (\ref{T3.8}) becomes
$\frac{\displaystyle\partial f_3}{\displaystyle\partial x^2}=0$. It follows that $f_3$ is constant on $I_1\times I_2$, which satisfies (\ref{T3.8}).
\end{proof}

\begin{example}
The sequential warped product manifold $$\left(I_1\times_{f}I_2\right)\times_{h}\mathbb R=\Big(I_1\times I_2\times \mathbb R,\ g=(dx^1)^2+f^2(dx^2)^2+h^2(dx^3)^2\Big),$$ for
$$f(x^1)=x^1+m, \ \ h(x^1,x^2)=(x^1+m)\cos (x^2+n),$$
with $m, n\in \mathbb R$ such that $x^1+m>0$, and $\left|x^2+n\right|<\displaystyle\frac{\pi}{2}$ for any $x^1\in I_1$, $x^2\in I_2$, is a flat Riemannian manifold.
\end{example}

And we can further deduce
\begin{corollary}\label{corT3.8b}
There do not exist proper flat sequential warped product manifolds of the form
$$\left(\mathbb R\times_{f}\mathbb R\right)\times_{h}\mathbb R=\Big(\mathbb R^3,\ g=(dx^1)^2+f^2(dx^2)^2+h^2(dx^3)^2\Big).$$
\end{corollary}

\begin{theorem}\label{ps11ma}
If $f_1=f_2=f(x^3)$, $f_3=f_3(x^1,x^2)$, then the doubly warped product manifold
$$_{\frac{1}{f}}(I_1\times I_2)\times_{\frac{1}{f_3}}I_3=(I,g)$$
is a flat Riemannian manifold if and only if
$$f=k\in \mathbb R\setminus \{0\}\ \textrm{ and }f_3(x^1,x^2)=\frac{\displaystyle 1}{\displaystyle c_1x^1+c_2x^2+c_3},$$
with $c_1, c_2, c_3\in \mathbb R$ such that $c_1x^1+c_2x^2+c_3\neq 0$ for any $(x^1,x^2)\in I_1\times I_2$.

If $c_1=c_2=0$, then the manifold is just a direct product, and the manifold reduces to a warped product manifold in the rest.
\end{theorem}
\begin{proof}
We have: $a_{13}=f_3\frac{\displaystyle f'}{\displaystyle f}=a_{23}$, $a_{31}=\frac{\displaystyle f}{\displaystyle f_3}\cdot\frac{\displaystyle \partial f_3}{\displaystyle \partial x^1}$, $a_{32}=\frac{\displaystyle f}{\displaystyle f_3}\cdot\frac{\displaystyle \partial f_3}{\displaystyle \partial x^2}$, and $a_{12}=a_{21}=0$, and we get:
\begin{align*}
\R(E_1,E_1)&=E_1(a_{31})+E_3(a_{13})-a_{31}^2-a_{13}^2-a_{13}a_{23},\\
\R(E_2,E_2)&=E_2(a_{32})+E_3(a_{23})-a_{32}^2-a_{23}^2-a_{13}a_{23},\\
\R(E_3,E_3)&=E_1(a_{31})+E_2(a_{32})+E_3(a_{13})+E_3(a_{23})-a_{31}^2-a_{13}^2-a_{32}^2-a_{23}^2,\\
\R(E_1,E_2)&=E_1(a_{32})-a_{31}a_{32},\\
\R(E_1,E_3)&=E_1(a_{23}),\\
\R(E_2,E_3)&=E_2(a_{13}).
\end{align*}

Then, $\R=0$ if and only if
\begin{equation*}\label{T3.11}
\left\{
\begin{aligned}
&f'=0\\
& \frac{\displaystyle \partial }{\displaystyle \partial x^1}\left(\frac{\displaystyle 1}{\displaystyle f_3}\cdot\frac{\displaystyle \partial f_3}{\displaystyle \partial x^1}\right)=\left(\frac{\displaystyle 1}{\displaystyle  f_3}\cdot\frac{\displaystyle \partial f_3}{\displaystyle \partial x^1}\right)^2\\
& \frac{\displaystyle \partial }{\displaystyle \partial x^2}\left(\frac{\displaystyle 1}{\displaystyle f_3}\cdot\frac{\displaystyle \partial f_3}{\displaystyle \partial x^2}\right)=\left(\frac{\displaystyle 1}{\displaystyle  f_3}\cdot\frac{\displaystyle \partial f_3}{\displaystyle \partial x^2}\right)^2\\
& \displaystyle \frac{\partial }{\partial x^1}\left(\displaystyle \frac{1}{f_3}\cdot\frac{\partial f_3}{\partial x^2}\right)=\left(\displaystyle \frac{1}{f_3}\cdot\frac{\partial f_3}{\partial x^1}\right)\left(\displaystyle \frac{1}{f_3}\cdot\frac{\partial f_3}{\partial x^2}\right)
\end{aligned}
\right. ,
\end{equation*}
which implies that $f$ is constant. We get the expression of $f_3$ with the same proof as for Theorem \ref{ps11mac7}.
\end{proof}

And we can further deduce
\begin{corollary}\label{corT3.8ba}
There do not exist proper flat doubly warped product manifolds of the form
$$_{f}\mathbb R^2\times_{h}\mathbb R=\Big(\mathbb R^3,\ g=f^2[(dx^1)^2+(dx^2)^2]+h^2(dx^3)^2\Big).$$
\end{corollary}

\section*{Declarations}
{\bf Competing interests:} The authors have no competing interests to declare that are relevant to the content of this article.

{\bf Funding:} No funding was received for conducting this study.
%\\
%\\
%{\bf Authors' contribution:} A.M.B. and D.R.L. wrote and approved the final version of the manuscript.

%item Data availability: Not applicable.


\begin{thebibliography}{00}

\bibitem{bishop}  Bishop, R.L., O'Neill, B.: Manifolds of negative curvature. Trans. Am. Math. Soc. 145, 1--49 (1969).
https://doi.org/10.2307/1995057

\bibitem{bl22}  Blaga, A.M.: On some $3$-dimensional almost $\eta$-Ricci solitons with diagonal metrics. Filomat 38(31), 10935--10961 (2024)
%https://doi.org/10.2298/FIL2431935B.

\bibitem{ude}  De, U.C., Shenawy, S., \"Unal, B.: Sequential warped products: curvature and conformal vector fields. Filomat 33(13), 4071--4083 (2019).
https://doi.org/10.2298/FIL1913071D

\bibitem{erlich}  Ehrlich, P.E.: Metric deformations of Ricci and sectional curvature on compact Riemannian manifolds. Ph.D. dissertation, SUNY, Stony Brook, N.Y. (1974)

\bibitem{neil}  O'Neill, B.: Semi-Riemannian geometry with applications to relativity. Pure and Applied Mathematics 103, Academic Press, New York (1983)

\bibitem{nolker}  N\"olker, S.: Isometric immersions of warped products. Diff. Geom. Appl. 6(1), 1--30 (1996).
https://doi.org/10.1016/0926-2245(96)00004-6

\bibitem{hakan}  Ta\c stan, H.M.: Biwarped product submanifolds of a K\" ahler manifold. Filomat 2(7), 2349--2365 (2018).
https://doi.org/10.2298/FIL1807349T
%
\end{thebibliography}
\end{document}